\documentclass[3p,12pt]{elsarticle}
\usepackage{amsmath,amssymb,amsfonts,verbatim}
\journal{Journal of Algebra}
\newtheorem{theorem}{Theorem}[section]

\newtheorem{prop}[theorem]{Proposition}
\newtheorem{cor}[theorem]{Corollary}

\newenvironment{proof}{{\flushleft\it Proof.}}{\hfill$\square$ \\}
\numberwithin{equation}{section}

\def\A{\mathcal{A\ts}}
\def\Bb{\ts\,\overline{\!\Br\!}\,\ts}
\def\ad{\operatorname{ad}}
\def\al{\alpha}
\def\alh{\widehat{\alpha}}

\def\Bb{\,\overline{\hspace{-0.75pt}\Br\hspace{-0.75pt}}\,}
\def\be{\beta}
\def\beh{\widehat{\beta}}
\def\Br{\mathrm{B}}

\def\C{\mathfrak{C}}
\def\CC{\mathbb{C}}
\def\comp{\ts{\scriptstyle\circ}\ts}

\def\d{\partial}
\def\de{\delta}

\def\E{\mathcal{E}}
\def\End{\operatorname{End}\ts}
\def\ep{\varepsilon}

\def\F{\mathcal{F}}

\def\gl{\mathfrak{gl}}
\def\glhat{\widehat{\mathfrak{gl}}}

\def\h{\mathfrak{h}}
\def\H{\mathfrak{H}}
\def\Hb{\Ub(\,\hh\,)}
\def\hh{\widehat{\mathfrak{h}}}
\def\Hr{\Ur(\,\hh\,)}

\def\Ib{{\bar\Ir}}
\def\id{{\rm id}}
\def\Ind{\operatorname{Ind}}

\def\Ir{\mathrm{I}}

\def\Jb{{\,\overline{\hspace{-1.25pt}\Jr}}}
\def\Jpr{\Jr^{\,\prime}} 
\def\Jpb{\Jb{}^{\,\prime}} 
\def\Jr{{\mathrm J}}

\def\ka{\kappa}

\def\la{\lambda}

\def\lcd{\ts,\ldots,}
\def\le{\leqslant}

\def\muh{\widehat{\mu}}
\def\n{\mathfrak{n}}
\def\np{\n^{\ts\prime}}
\def\nt{\widehat{\mathfrak{n}}}
\def\ntp{\widehat{\mathfrak{n}}^{\,\ts\prime}}

\def\om{\omega}
\def\op{\oplus}
\def\ot{\otimes}

\def\p{\mathfrak{p}}
\def\P{\mathfrak{P}}

\def\q{\mathfrak{q}}

\def\R{\mathfrak{R}\ts}

\def\si{\sigma}
\def\sl{\mathfrak{sl}}
\def\slhat{\widehat{\mathfrak{sl}}}
\def\Sym{\mathfrak{S}}

\def\t{\mathfrak{t}}
\def\T{\mathfrak{T}}
\def\th{\hat{\mathfrak{t}}}
\def\ts{\hskip.75pt}
\def\tts{\hskip.5pt}

\def\Ub{\ts\overline{\hspace{-0.75pt}\Ur\hspace{-0.75pt}}\ts}
\def\up{\upsilon}
\def\Ur{\mathrm{U}}

\def\ze{\zeta}
\def\ZZ{\mathbb{Z}}

\begin{document}

\begin{frontmatter}

\title{Cherednik algebras and Zhelobenko operators}
 
\author[a,b]{Sergey Khoroshkin\,}
\address[a]{Institute for Theoretical and Experimental Physics, 
%B.\,Cheremushkinskaya 25,
Moscow 117218, Russia}
\address[b]{National Research University Higher School of Economics, 
%Myasnitskaya 20,
Moscow 101000, Russia}

\author[c]{Maxim Nazarov\,}
\address[c]{%Department of Mathematics,  
University of York, 
York YO10 5DD, England
\\[20pt]
{\rm\normalsize 
To Professor Alexandre Kirillov on the occasion of his eightieth birthday}
\\[-10pt]
}

%\begin{keyword}
%affine Lie algebras, 
%Cherednik algebras, 
%Zhelobenko operators
%\MSC[2000]17B10
%\end{keyword}

%\begin{abstract}
%We study canonical intertwining operators
%between modules of the trigonometric Cherednik algebra,
%induced from the standard modules of the degenerate affine Hecke algebra.
%We~show that these operators correspond to the Zhelobenko operators
%for the affine Lie algebra $\slhat_m\,$.
%To establish the correspondence, we use the 
%functor of Arakawa, Suzuki and Tsuchiya
%which maps certain $\slhat_m$-modules 
%to modules of the Cherednik algebra. 
%\end{abstract}

\end{frontmatter}

\thispagestyle{empty}%%%%%%%%%%%%%%%%%%%%%%%%%%%%%%%%%%%%%%%%%%%%%%%%%%%%%%%%%

%=============================================================================

\vspace{-4pt}
\section*{Introduction}

%-----------------------------------------------------------------------------

\subsection{}
\label{sec:01}

In the present article we study 
the trigonometric Cherednik algebra $\C_N$
corresponding to the general linear Lie algebra $\gl_N\,$.
The complex associative algebra $\C_N$ is generated by 
the symmetric group ring $\CC\ts\Sym_N\,$, 
by the ring $\P_N$ of Laurent polynomials in $N$ variables $x_1\lcd x_N$
and by another family of pairwise commuting elements 
denoted by $u_1\lcd u_N\,$.
The subalgebra of $\C_N$ generated by the first two rings is %just
the crossed product 
$\Sym_N\ltimes\P_N$ where the symmetric group $\Sym_N$ 
permutes the variables $x_1\lcd x_N\,$. The subalgebra generated by 
$\Sym_N$ and $u_1\lcd u_N$ is the degenerate affine Hecke algebra
$\H_N$ introduced by Drinfeld \cite{D1} and Lusztig \cite{L}.
The other defining relations in $\C_N$ are given in the beginning
of Section 2 of our article.
In particular, the algebra $\C_N$ depends on a parameter $\ka\in\CC\,$.

The degenerate affine Hecke algebra 
$\H_N$ has a distinguished family
of modules which~are called \emph{standard\/}. 
These modules are determined by pairs of
sequences $\la=(\ts\la_1\lcd\la_m\ts)$ and $\mu=(\ts\mu_1\lcd\mu_m\ts)$
of length $m$ of complex numbers such that
for every $a=1\lcd m$ the difference $\la_{\ts a}-\mu_{\ts a}$ 
is a positive integer, while
$
\la_{\ts 1}-\mu_{\ts 1}+\ldots+\la_{\ts m}-\mu_{\ts m}=N\ts.
$
We denote the corresponding standard module of $\H_N$ by 
$S_{\mu}^{\ts\la}\,$. 
It is induced
from a one-dimensional module of the subalgebra of $\H_N$ 
generated by $u_1\lcd u_N$ 
and by the subgroup of $\Sym_N$ 
preserving the partition of the sequence $1\lcd N$ 
to consecutive segments
of lengths $\la_{\ts 1}-\mu_{\ts 1}\lcd\la_{\ts m}-\mu_{\ts m}\,$. 
This subgroup of $\Sym_N$ acts on the one-dimensional module trivially, 
while any generator $u_{\ts p}$ acts as
$
\mu_a-a+h
$
where $a$ is the number of the segment of the sequence $1\lcd N$ which 
the index $p$ belongs to, and $h$ is the number of the place
of the index $p$ within that segment. 

%It is known \cite{R} that every irreducible module of $\H_N$ is a
%quotient of a standard module. 

Now consider the symmetric group $\Sym_m$ which acts
on sequences of length $m$ of complex numbers by permutations.
We will denote by the symbol $\comp$ the corresponding
\emph{shifted} action of $\Sym_m\,$.
To define the latter action,
one takes a sequence of length $m\,$, subtracts the sequence 
$(1\ts,\ldots,m)$ from it, permutes the resulting sequence 
and adds the sequence $(1\ts,\ldots,m)$ back.
If $\la_a-\la_b\notin\ZZ$ or equivalently $\mu_a-\mu_b\notin\ZZ$
for all $a\neq b\,$, then the standard $\H_N\ts$-module $S_{\mu}^{\ts\la}$
is irreducible. Moreover, then for every permutation $\si\in\Sym_m$
the standard module $S_{\ts\si\comp\mu}^{\,\si\comp\la}$
is isomorphic to $S_{\mu}^{\ts\la}\,$. Hence there exists 
an intertwining mapping 
$\,S_{\mu}^{\ts\la}\to S_{\ts\si\comp\mu}^{\,\si\comp\la}\,$
of $\H_N\ts$-modules, unique up to scalar multiplier.
These mappings were already used by Rogawski in \cite{R}.

The standard $\H_N\ts$-module $S_{\mu}^{\ts\la}$ has another %remarkable
realization due to Arakawa, Suzuki and Tsuchiya \cite{AST}.
Take the complex %general linear 
Lie algebra $\gl_m\,$. The sequences of length $m$ of complex numbers 
can be regarded as weights of $\gl_m\,$.
There we employ the Cartan subalgebra $\t$ of~$\gl_m$ described in 
Subsection \ref{sec:13}. Then
the above defined action $\comp$
becomes the shifted action of the Weyl group $\Sym_m$
of $\gl_m$ on weights.
Take the Verma module $M_\mu$ of $\gl_m$
corresponding~to~the~weight~$\mu\,$.

The tensor product $(\CC^{\ts m})^{\ot N}\ot\,M_\mu$
of $\gl_m\ts$-modules can be equipped with an action of 
the algebra $\H_N\,$, which commutes with the action of $\gl_m\,$.
The symmetric group $\Sym_N\subset\H_N$ acts on 
$(\CC^{\ts m})^{\ot N}\ot\,M_\mu$
by permutations of the $N$ tensor factors $\CC^{\ts m}\ts$,
while the elements \eqref{zp} of $\H_N$ act as the operators \eqref{yact}
respectively. 
Let $\n$ be the %maximal 
nilpotent subalgebra of $\gl_m$ defined in Subsection \ref{sec:13}.
The space 
$
(\ts(\CC^{\ts m})^{\ot N}\ot\,M_\mu\ts)\ts)_{\,\n}^{\ts\la}
$
of $\n\ts$-coinvariants of weight $\la$ inherits an action of 
the algebra $\H_N\,$. As a $\H_N\ts$-module
it is isomorphic to $S_{\mu}^{\ts\la}\,$,
see our Proposition~\ref{1.3}. 

Following Zhelobenko \cite{Z}, for any $\la$ and $\mu$
obeying the above non-integrality conditions, and 
for any permutation $\si\in\Sym_m\,$,
one can define a canonical linear map
\begin{equation}
\label{MM}
(\ts(\CC^{\ts m})^{\ot N}\ot\,M_\mu\ts)\ts)_{\,\n}^{\ts\la}
\,\to
(\ts(\CC^{\ts m})^{\ot N}\ot\,
M_{\ts\si\comp\mu}\ts)\ts)_{\,\n}^{\,\si\comp\la}\,.
\end{equation}
For this definition and its generalizations
see the work of Khoroshkin and Ogievetsky
\cite{KO}. In particular,
the linear map \eqref{MM} is $\H_N\ts$-intertwining. 
Using Proposition~\ref{1.3},
the~map \eqref{MM} determines 
an $\H_N\ts$-intertwining map
$\,S_{\mu}^{\ts\la}\to S_{\ts\si\comp\mu}^{\,\si\comp\la}\,$.
By the irreducibility of the source and of target 
standard $\H_N\ts$-modules here, the latter map
coincides with the intertwining map from \cite{R}
up to a scalar multiplier.

We will work with the special linear Lie algebra
$\sl_m$ alongside of $\gl_m\ts$. In particular,
our $\n$ is a subalgebra of $\sl_m\,$. Further, 
the Cartan subalgebra $\h$ of $\sl_m$ described in 
Subsection~\ref{sec:13} is contained in $\t\subset\gl_m\,$.
%Throughout this article we will
Let us denote by $\al$ and $\beta$ the weights of
$\sl_m$ corresponding to the weights $\la$ and $\mu$ of
$\gl_m$ by restriction. Hence $\al$ and $\be$ are elements of
the space dual to~$\h\,$.

The Verma module $M_\be$ is isomorphic to the restriction of
$M_\mu$ to subalgebra $\sl_m\subset\gl_m\,$. 
However, another action of $\H_N$ on $(\CC^{\ts m})^{\ot N}\ot\,M_\be$
can be defined by using only the structure of $M_\be$
as a module of $\sl_m\,$. Namely, 
he symmetric group $\Sym_N\subset\H_N$ acts on 
$(\CC^{\ts m})^{\ot N}\ot\,M_\be$ again
by permutations of the $N$ tensor factors $\CC^{\ts m}\ts$,
but the elements \eqref{zp} of $\H_N$ act as the operators \eqref{yacts}
respectively. The space 
$
(\ts(\CC^{\ts m})^{\ot N}\ot\,M_\be\ts)\ts)_{\,\n}^{\ts\al}
$
of $\n\ts$-coinvariants of weight $\al$ inherits an action of %the algebra 
$\H_N\,$. As a $\H_N\ts$-module
it is isomorphic to the pullback of $S_{\mu}^{\ts\la}$
relative to the automorphism \eqref{shift} of $\H_N$
where $f=-\,(\ts\mu_1+\ldots+\mu_m\ts)\ts/\ts m\,$,
see Corollary~\ref{1.4}. This automorphism
acts trivially on the elements of the subalgebra $\Sym_m\subset\H_N\,$.
Note that the latter  space of $\n\ts$-coinvariants
can be naturally identified with the space
$
(\ts(\CC^{\ts m})^{\ot N}\ot\,M_\mu\ts)\ts)_{\,\n}^{\ts\la}\,\ts.
$

The shifted action of the group $\Sym_m$ on the weights of $\gl_m$
factors to an action on the weights of $\sl_m\,$. 
By again following \cite{KO} and \cite{Z},
one can also define a canonical linear map
\begin{equation}
\label{MMS}
(\ts(\CC^{\ts m})^{\ot N}\ot\,M_\be\ts)\ts)_{\,\n}^{\ts\al}
\,\to
(\ts(\CC^{\ts m})^{\ot N}\ot\,
M_{\ts\si\comp\be}\ts)\ts)_{\,\n}^{\,\si\comp\al}\,.
\end{equation}
If we identify the source vector spaces
of the maps \eqref{MM} and \eqref{MMS} as above, and also 
identify the target vector spaces,
then the two maps become the same. 
Note that the action $\comp$ of $\Sym_m$ on %the sequence 
$\mu$ preserves the sum $\mu_1+\ldots+\mu_m\,$.
Hence the map \eqref{MMS} is also $\H_N\ts$-intertwining.

\enlargethispage{24pt}%%%%%%%%%%%%%%%%%%%%%%%%%%%%%%%%%%%%%%%%%%%%%%%%%%%%%%%%

Via the Drinfeld duality
between $\H_N\ts$-modules and
modules of the Yangians of general linear Lie algebras \cite{D1}
this interpretation of intertwining maps for
the standard modules of $\H_N$
goes back to the work of Tarasov and Varchenko~\cite{TV}, 
see also our work \cite{KN1}. 
In the present work we extend
this interpretation to intertwining maps for
certain $\C_N\ts$-modules.
Instead of $\gl_m$ and $\sl_m$ above, we will use
the corresponding affine Lie algebras $\glhat_m$ and $\slhat_m\,$.

\newpage%%%%%%%%%%%%%%%%%%%%%%%%%%%%%%%%%%%%%%%%%%%%%%%%%%%%%%%%%%%%%%%%%%%%%%

%-----------------------------------------------------------------------------

\subsection{}
\label{sec:02}

We will regard $\slhat_m$ as a one-dimensional
central extension of the current Lie algebra
$\sl_m\ts[\,t,t^{-1}\ts]\,$.
We choose a basis element $C$
in the extending one-dimensional vector space.
For any $\ell\in\CC\,$, a module of $\slhat_m$
is said to be of level $\ell$ if $C$ acts as 
the scalar $\ell$ on that module. 
Let us extend the Cartan subalgebra $\h$ of $\sl_m$ 
by the one-dimensional space spanned by $C\ts$,
and denote by $\hh$ the Abelian subalgebra
of $\slhat_m$ so obtained.
For $\ell=\ka-m$ we will denote by $\alh$ and $\beh$
the extensions of the weights $\al$ and $\be$ from $\h$ to $\hh\,$, 
determined by setting $\alh\ts(C)=\beh\ts(C)=\ell\,$.
We will use the Verma module $M_{\ts\beh}$ of %the Lie algebra 
$\slhat_m$ as defined in Subsection~\ref{sec:23}.
\\[-14pt]

Let us now regard $\H_N$ as a subalgebra of $\C_N\,$.
Denote by ${\widehat{S}}_{\ts\mu}^{\,\la}$ the module of $\C_N$
induced from the standard module ${S}_{\ts\mu}^{\,\la}$ of $\H_N\,$.
The induced module also has another realization \cite{AST}.
Take the vector space $\P_N\ot(\CC^{\ts m})^{\ot N}\,$. 
It can be naturally identified with the tensor product of $N$ copies
of the vector space $\CC^{\ts m}[\ts t\ts,t^{-1}\ts]\,$.
By regarding the latter space as a module of~$\slhat_m$ of level zero,  
$\P_N\ot(\CC^{\ts m})^{\ot N}$ %also 
becomes a zero level module of $\slhat_m\,$. Further, the vector space
\begin{equation}
\label{PCM}
\P_N\ot(\CC^{\ts m})^{\ot N}\ot M_{\ts\beh}
\end{equation}
can be equipped with an action of the algebra $\C_N\,$. 
The symmetric group $\Sym_N\subset\C_N$ acts on \eqref{PCM}
by simultaneous permutations of the variables $x_1\lcd x_N$ and
of the $N$ tensor factors $\CC^{\ts m}$
while the subalgebra $\P_N\subset\C_N$ acts on \eqref{PCM}
via multiplication in the first tensor factor.
The elements \eqref{zp} of $\H_N\subset\C_N$ act on \eqref{PCM}
as the operators \eqref{ups} respectively.

In general, the action of $\C_N$ on the vector space \eqref{PCM}
does \emph{not\/} commute with that of~$\slhat_m\,$.
However, let $\nt$ be the
nilpotent subalgebra of $\slhat_m$ defined in Subsection \ref{sec:21}.
For $\ell=\ka-m$ the action of $\C_N$ on \eqref{PCM}
preserves the image of the action of $\ts\nt\,$. %see Corollary \ref{2.2}. 
Therefore the space 
\begin{equation}
\label{PCMAB}
(\,\P_N\ot(\CC^{\ts m})^{\ot N}\ot M_{\ts\beh}\,)_{\,\ts\nt}^{\,\alh}
\end{equation}
of $\nt\ts$-coinvariants of \eqref{PCM}
of weight $\ts\alh\ts$ inherits an action of %the algebra 
$\C_N\,$. The automorphism \eqref{shift} of $\H_N$ 
extends to $\C_N$ so that it acts on the elements
of the subalgebra $\P_N\subset\C_N$ trivially.
As a $\C_N\ts$-module, \eqref{PCMAB}
is isomorphic to the pullback of ${\widehat{S}}_{\ts\mu}^{\,\la}$
relative to the extended automorphism \eqref{shift}
where $f=-\,(\ts\mu_1+\ldots+\mu_m\ts)\ts/\ts m\,$,
see our Corollary \ref{2.2} and Proposition~\ref{2.3}. 

Now consider the semidirect group product $\Sym_m\ltimes\ZZ^m\,$. 
%We will denote it by $\R_m\,$.
Extend the permutation action of the group $\Sym_m$ on sequences
of length $m$ of complex numbers to an action 
of $\Sym_m\ltimes\ZZ^m$
so that the elements of
$\ts\ZZ^m$ act by addition of the respective elements of
$\ts\ell\,\ZZ^m\subset\CC^m\,$.
Here we set $\ell=\ka-m$ as above.
The action $\comp$ of $\Sym_m$
on the sequences also extends to an action of the group
$\Sym_m\ltimes\ZZ^m\,$, where 
the elements of $\ZZ^m$ however
act by addition of the respective elements of
$\ts\ka\,\ZZ^m\subset\CC^m\,$.
Let us denote by the symbol $\comp$
the latter action of $\Sym_m\ltimes\ZZ^m$ 
on the sequences.
If $\la_a-\la_b\notin\ZZ+\ka\,\ZZ$ or equivalently 
$\mu_a-\mu_b\notin\ZZ+\ka\,\ZZ$
for $a\neq b\,$,  then the induced
$\C_N\ts$-module ${\widehat{S}}_{\ts\mu}^{\,\la}$ is irreducible. 
Moreover, then for every element 
$\om\in\Sym_m\ltimes\ZZ^m$ the 
\\[-2pt]%%%%%%%%%%%%%%%%%%%%%%%%%%%%%%%%%%%%%%%%%%%%%%%%%%%%%%%%%%%%%%%%%%%%% 
$\C_N\ts$-module ${\widehat{S}}_{\ts\om\ts\comp\mu}^{\,\om\ts\comp\la}$
is isomorphic to ${\widehat{S}}_{\mu}^{\ts\la}\,$.
Hence there is an intertwining mapping 
$
\widehat{S}_{\mu}^{\ts\la}
\to
\widehat{S}_{\,\om\ts\comp\mu}^{\ts\,\om\ts\comp\la}
$
of $\C_N\ts$-modules, unique up to a scalar multiplier.
These mappings were used by Suzuki in~\cite{S1}.
Recently they were further used by Balagovi\'c in \cite{B}.

\enlargethispage{4pt}%%%%%%%%%%%%%%%%%%%%%%%%%%%%%%%%%%%%%%%%%%%%%%%%%%%%%%

The group $\Sym_m\ltimes\ZZ^m$ is isomorphic to
the extended affine Weyl group of %the reductive Lie algebra 
$\gl_m\,$, for details
see Subsection~\ref{sec:24}.
This group is $\ZZ\ts$-graded 
so that the degree of any element of
$\Sym_m$ is zero, while
the degree of any element of $\ZZ^m$ is
the sum of its $m$ components.
All the elements of degree zero 
make a subgroup of $\Sym_m\ltimes\ZZ^m$ isomorphic to 
the proper affine Weyl group of~$\gl_m\,$.
Note that this subgroup is also isomorphic to the Weyl group
of the affine Lie algebra $\slhat_m\,$.

\newpage%%%%%%%%%%%%%%%%%%%%%%%%%%%%%%%%%%%%%%%%%%%%%%%%%%%%%%%%%%%%%%%%%%%%%

Again regard $\la$ and $\mu$ as weights of $\gl_m\,$. 
Restrict them to the weights $\al$ and $\be$ of $\sl_m\,$.
Extend the latter two to the weights $\alh$ and $\beh$ of $\slhat_m$
as above. 
%In particular, here we use the value $\ell=\ka-m\,$.
The action $\comp$ of %the group 
$\Sym_m\ltimes\ZZ^m$ on $\la$ and $\mu$ 
determines its action on $\alh$ and $\beh\,$.
We will still denote by $\comp$ the action of $\Sym_m\ltimes\ZZ^m$ 
so determined.
It can also be described as a shifted action of 
the group $\Sym_m\ltimes\ZZ^m$
on those weights of $\slhat_m$ which take the value $\ell=\ka-m$
at $C\in\hh\,$, see Subsection \ref{sec:31} for details.

By following \cite{KO} and \cite{Z},
for every element $\om\in\Sym_m\ltimes\ZZ^m$ 
we can define a canonical linear map 
from the vector space \eqref{PCMAB} to the vector space 
\begin{equation}
\label{PCMOM}
(\,\P_N\ot(\CC^{\ts m})^{\ot N}\ot 
M_{\ts\om\ts\comp\beh}\,)_{\,\ts\nt}^{\,\om\ts\comp\alh}\,.
\end{equation}
Details of this definition are given in our Subsection \ref{sec:41}.
%of the present article. 
Denote by $g$ the $\ZZ\ts$-degree of~$\om\,$.
Our linear map is $\C_N\ts$-intertwining only if
$\ka=0$ or $g=0\,$.
In general, it becomes $\C_N\ts$-intertwining
if we pull the action of $\C_N$ on the target space \eqref{PCMOM}
back through the automorphism \eqref{shift} where $f=\ka\,g\ts/m\,$. 
Here we use Proposition \ref{2.4} and its Corollary \ref{2.5}
which seem to be new.

By using Proposition \ref{2.3}
we can now replace the source and the target modules of this 
$\C_N\ts$-intertwining linear map by their isomorphic modules.
The source module can be replaced by
the pullback of ${\widehat{S}}_{\ts\mu}^{\,\la}$
relative to the automorphism \eqref{shift}
where $f=-\,(\ts\mu_1+\ldots+\mu_m\ts)\ts/\ts m\,$.
Note that the sum of the terms of the
sequence $\om\ts\comp\ts\mu$ is equal to
$\mu_1+\ldots+\mu_m+\ka\,g$
by the definition of the action $\comp$ of %the group 
$\Sym_m\ltimes\ZZ^m$ on the sequences. 
Therefore the target module here can be replaced by the pullback of
$\widehat{S}_{\,\om\ts\comp\mu}^{\ts\,\om\ts\comp\la}$
relative to the automorphism \eqref{shift} where 
$$
f=-\,(\ts\mu_1+\ldots+\mu_m+\ka\,g\ts)\ts/\ts m+\ka\,g\ts/m
=-\,(\ts\mu_1+\ldots+\mu_m\ts)\ts/\ts m\,.
$$
Since the values of $f$ for the 
the source and the target modules are the same, 
our canonical linear map from 
\eqref{PCMAB} to \eqref{PCMOM} 
also determines a $\C_N\ts$-intertwining linear map
$
\widehat{S}_{\mu}^{\ts\la}
\to
\widehat{S}_{\,\om\ts\comp\mu}^{\ts\,\om\ts\comp\la}\,.
$
By the irreducibility of the source and of target 
induced $\C_N\ts$-modules here, the latter map
coincides with the intertwining map from \cite{S1}
up to a scalar multiplier.

Thus the principal result of our article is
that the Zhelobenko operators for the affine Lie algebra
$\slhat_m$ determine intertwining maps between induced modules
of the trigonometric Cherednik algebra $\C_N\,$.
Note that these operators have been defined in \cite{Z}
only for complex reductive Lie algebras.
Hence we also extend the definition of \cite{Z}
to affine Lie algebras.

%-----------------------------------------------------------------------------

\subsection{}
\label{sec:03}

Let us now briefly survey our article. In Section \ref{sec:1} we collect 
some basic facts about the degenerate affine Hecke algebra $\H_N\,$,
including the realisation of standard modules from~\cite{AST}.
In Section \ref{sec:2} we recall the definition of the trigonometric 
Cherednik algebra $\C_N\,$, and describe the action of $\C_N$ on the 
spaces of $\nt\ts$-coinvariants. By using this action, we give
the realisation of induced modules of $\C_N$ mentioned above.
Towards the end of Section \ref{sec:2} we introduce the extended
affine Weyl group of $\gl_m\,$, and describe its action on the 
spaces of $\nt\ts$-coinvariants. Our Proposition \ref{2.4}
relates this action to the action of $\C_N$ on the same spaces.
This relation is the key technical result of our article. 
In Section \ref{sec:3} we define the Zhelobenko operators~for 
the affine Lie algebra $\slhat_m\,$. Theorem \ref{3.6}
relates these operators to the algebra $\C_N\,$.
Further details of this relation are worked out in Section \ref{sec:4}. 

The first named author has been supported 
by the Russian Academic Excellence Project 5\ts-100
and by the RSF grant 16-11-10316.
%dated 11.05.2016.
The second named author was supported by 
the EPSRC grant N023919 and by
a Santander International Connections Award.

\newpage%%%%%%%%%%%%%%%%%%%%%%%%%%%%%%%%%%%%%%%%%%%%%%%%%%%%%%%%%%%%%%%%%%%%%%

%=============================================================================

\section{Hecke algebras}
\label{sec:1}
\medskip

%-----------------------------------------------------------------------------

\subsection{}
\label{sec:11}

We begin with the definition of the
\textit{degenerate affine Hecke algebra\/} $\H_N\ts$
corresponding to the general linear group $\mathrm{GL}_N$ over a local
non-Archimedean field. This algebra has been introduced by Drinfeld
[D2], see also the work of Lusztig [L]. 
The complex associative algebra $\H_N$ is
generated by the symmetric group algebra $\CC\ts\Sym_N$ and by pairwise
commuting elements $u_1\lcd u_N$ with the cross relations for $p=1\lcd N-1$
and $q=1\lcd N$
\begin{align*}
%\label{cross1}
\si_{p}\,u_q&=u_q\,\si_{p}
\quad\text{for}\quad
q\neq p\ts,p+1\,;
\\
%\label{cross2}
\si_{p}\,u_p&=u_{p+1}\,\si_{p}-1\,.
\end{align*}
Here and in what follows $\si_p\in \Sym_N$ denotes the 
transposition of numbers $p$ and $p+1\ts$. 
More generally, $\si_{pq}\in \Sym_N$ will denote the 
transposition of the numbers $p$ and $q\ts$. The group algebra
$\CC\ts\Sym_N$ can be then regarded as a subalgebra in $\H_N\ts$.
Furthermore, it follows from the defining relations of $\H_N$
that a homomorphism $\H_N\to\CC\ts\Sym_N\ts$, identical
on the subalgebra $\CC\ts\Sym_N\subset \H_N\ts$,
can be defined by the assignments
\begin{equation}
\label{evalhecke}
u_p\mapsto\si_{1p}+\ldots+\si_{p-1,p}
\quad\text{for}\quad
p=1\lcd N\ts.
\end{equation}

We will also use the elements of the algebra $\H_N\ts$
\begin{equation}
\label{zp}
z_p=u_p-\si_{1p}-\ldots-\si_{p-1,p}
\quad\text{where}\quad
p=1\lcd N\ts.
\end{equation}
Notice that $z_p\mapsto0$ under the homomorphism $\H_N\to\CC\ts\Sym_N\ts$
defined by \eqref{evalhecke}. For every permutation $\si\in\Sym_N$ we have
\begin{equation}
\label{yrel}
\si\,z_p\,\si^{-1}=z_{\ts\si(p)}\,.
\end{equation}
The elements $z_1\lcd z_N$ do not commute, but satisfy
the commutation relations
\begin{equation}
\label{ycom}
[\,z_p\,,z_q\,]=\si_{pq}\,(z_p-z_q)\,.
\end{equation}
The relations \eqref{yrel} and \eqref{ycom} easily follow from
the above definition of the algebra $\H_N$, see for instance
\cite[Section 1]{KN1}.
Obviously, the algebra $\H_N$ is generated by $\CC\ts\Sym_N$
and the elements $z_1\lcd z_N\ts$.
Together with relations in $\CC\ts\Sym_N\ts$, 
\eqref{yrel} and \eqref{ycom} are defining relations of $\H_N$~too.

It immediately follows from the definition of $\H_N$ that
for any $f\in\CC$ an automorphism of this algebra,
identical on the subalgebra $\CC\ts\Sym_N\subset\H_N\ts$, 
can be defined by mapping 
\begin{equation}
\label{shift}
u_p\mapsto u_p+f
\quad\text{for}\quad
p=1\lcd N\ts.
\end{equation}
Note that then by \eqref{zp}
$$
z_p\mapsto z_p+f
\quad\text{for}\quad
p=1\lcd N\ts.
$$
By pulling the trivial one-dimensional module of $\CC\ts\Sym_N$ back through
the homomorphism~\eqref{evalhecke}, and further back through the automorphism
\eqref{shift}, we get a one-dimensional module of $\H_N\,$.
On the latter module
each of the elements $z_1\lcd z_N\in\H_N$ 
acts as multiplication by $f\ts$.
Let us denote this module by $S_f^{\hspace{0.8pt}f\ts+N}$,
this peculiar choice of notation will be justified next.

Fix a postive integer $m\ts$. Take any two sequences
$\la=(\ts\la_1\lcd\la_m\ts)$ and $\mu=(\ts\mu_1\lcd\mu_m\ts)$
of length $m$ of complex numbers.
For each $a=1\lcd m$ denote $\nu_a=\la_{\ts a}-\mu_{\ts a}$ and suppose that 
$\nu_a$ is a non-negative integer. 
Note that unlike in the Introduction,
here we allow the equality $\nu_a=0\,$.
We still suppose that $\nu_1+\ldots+\nu_m=N\ts$.
Denote $\nu=(\ts\nu_1\lcd \nu_m\ts)\,$.
%so that $\nu$ is a \textit{composition\/} of $N\ts$.
Let $\Sym_\nu$ be the corresponding subgroup of the symmetric 
group $\Sym_N\ts$. 
This subgroup is naturally isomorphic to the direct product
$\Sym_{\nu_1}\times\ldots\times\Sym_{\nu_m}\ts$.
%Further, 
The tensor product 
$\H_{\nu_1}\otimes\ldots\otimes\H_{\nu_m}$
can be naturally identified 
with the subalgebra of $\H_N$ generated by the subgroup
$\Sym_\nu\subset\Sym_N$ and by all the pairwise commuting elements
$u_1\lcd u_N\ts$. Denote by $\H_\nu$ this subalgebra.
The induced module of $\H_N$
\begin{equation*}
%\label{standard}
\Ind_{\,\H_\nu}^{\,\H_N}\,
S_{\mu_1}^{\ts\la_1}
\ot
S_{\mu_2-1}^{\ts\la_2-1}
\ot\ldots\ot 
S_{\mu_m-m+1}^{\ts\la_m-m+1}
\end{equation*}
is called \textit{standard\/} and will be denoted by $S_{\mu}^{\ts\la}\,$.
%throughout this article. 
In the particular case of $m=1$ and $\mu_1=f$
we have $\la_1=f+N$ and
$S_{\mu}^{\ts\la}=S_f^{\hspace{0.8pt}f\ts+N}$.
The reason to use in 
the definition of $S_{\mu}^{\ts\la}\,$
the numbers $\la_a-a+1$ and $\mu_a-a+1$
rather than $\la_a$ and $\mu_a$
will become clear in Subsection~\ref{sec:13}.

%-----------------------------------------------------------------------------

\subsection{}
\label{sec:12}

Let us now recall a construction due to Cherednik
\cite[Example 2.1]{C2}. It has been further developed by
Arakawa, Suzuki and Tsuchiya \cite[Subsection 5.3]{AST}. 
Let $U$ be any module over the complex general linear Lie algebra $\gl_m\ts$.
Let $E_{\ts ab}\in\gl_m$ with $a,b=1\lcd m$ be the
standard matrix units. We will also regard the matrix units $E_{\ts ab}$
as elements of the algebra $\End(\CC^{\ts m})$, this should not cause any 
confusion. Let us consider the tensor product $(\CC^{\ts m})^{\ot N}\ot\,U$
of $\gl_m\ts$-modules. Here each of the $N$ tensor factors $\CC^{\ts m}$ is
a copy of the natural $\gl_m$-module. We shall use the indices
$1\lcd N$ to label these $N$ tensor factors. 
For any index $p=1\lcd N$
we will denote by $E^{\ts(p)}_{ab}$ the operator on the vector space 
$(\CC^{\ts m})^{\ot N}$ acting as
\begin{equation}
\label{epab}
\id^{\ts\ot\ts(p-1)}\ot E_{\ts ab}\ot\id^{\ts\ot\ts(N-p)}\,.
\end{equation}

\begin{prop}
\label{1.1}
{\rm(i)}
Using the\/ $\gl_m\ts$-module structure of $U$,  
an action of the algebra $\H_N$ on the vector space 
$(\CC^{\ts m})^{\ot N}\ot\ts U$
is defined as follows: the symmetric 
group $\Sym_N\subset \H_N$ acts %naturally 
by permutations of the $N$ tensor factors\/ $\CC^{\ts m}$, and the element 
$z_p\in \H_N$ with $p=1\lcd N$~acts~as 
\begin{equation}
\label{yact}
\sum_{a,b=1}^m E_{\ts ab}^{\ts(p)}\ot E_{\ts ba}\,.
\end{equation}
{\rm(ii)} 
This action of $\H_N$ commutes with the diagonal action of\/ 
$\gl_m$ on $(\CC^{\ts m})^{\ot N}\ot\ts U\ts$.
\end{prop}

For a proof of this proposition see for instance \cite[Section 1]{KN1}.
By using Proposition \ref{1.1} we get a functor
$
\E_N: U\mapsto(\CC^{\ts m})^{\ot N}\ot U
$
from the category of all $\gl_m\ts$-modules 
to the category of bimodules over $\gl_m$ and $\H_N\ts$.
We will also use a version of this proposition for the 
special linear Lie algebra $\sl_m$ istead of $\gl_m\ts$.
Denote $I=E_{\ts 11}+\ldots+E_{\ts mm}$ 
so that $\gl_m=\sl_m\op\CC\,I\,$.
Moreover, we have
\begin{equation}
\label{i}
\sum_{a,b=1}^m 
E_{\ts ab}\ot E_{\ts ba}\,\in\,\frac1m\,\,I\ot I\,+\,\sl_m\ot\sl_m\,.
\end{equation}
Therefore an action of %the sum
\begin{equation}
\label{yacts}
\sum_{a,b=1}^m E_{\ts ab}^{\ts(p)}\ot E_{\ts ba}
-\frac1m\ \id^{\,\ot\ts N}\ot I
\end{equation}
can be defined on the vector space $(\CC^{\ts m})^{\ot N}\ot\,U$ 
by using only the $\sl_m\ts$-module structure of $U$.
Because the element $I\in\gl_m$ is central, 
the operators \eqref{yacts} with $p=1\lcd N$ satisfy the 
same commutation relations \eqref{ycom}
as the operators \eqref{yact} respectively instead of $z_1\lcd z_N$.
  
\begin{cor}
\label{1.2}
{\rm(i)}
By using the\/ $\sl_m\ts$-module structure of\/ $U$,  
an action of the algebra $\H_N$ on the vector space 
$(\CC^{\ts m})^{\ot N}\ot\ts U$
is defined as follows: the %symmetric 
group $\Sym_N\subset \H_N$ acts %naturally 
by permutations of the $N$ tensor factors\/ $\CC^{\ts m}$, and the element 
$z_p\in \H_N$ with $p=1\lcd N$~acts~as \eqref{yacts}.
\\
{\rm(ii)} 
This action of $\H_N$ commutes with the diagonal action of\/ 
$\sl_m$ on $(\CC^{\ts m})^{\ot N}\ot\ts U\ts$.
\end{cor}

Using Corollary \ref{1.2} we get a functor
$
\F_N: U\mapsto(\CC^{\ts m})^{\ot N}\ot U
$
from the category of all $\sl_m\ts$-modules 
to the category of bimodules over $\sl_m$ and $\H_N\ts$.
Our principal tool will be an analogue of
this functor for the affine Lie algebra $\slhat_m$ instead of 
$\sl_m\,$. The role of the degenerate affine algebra 
$\H_N$ will be then played by the 
{trigonometric Cherednik algebra\/}~$\C_N\,$.

%-----------------------------------------------------------------------------

\subsection{}
\label{sec:13}

Consider the \textit{triangular decomposition}
of the Lie algebra $\gl_m\ts$,
\begin{equation*}
%\label{tridec}
\gl_m=\n\op\t\op\np\ts.
\end{equation*}
Here $\t$ is the Cartan subalgebra of $\gl_m$ with the basis vectors
$E_{\ts 11}\,\lcd E_{\ts mm}\ts$. 
Every element of the vector space $\t^{\ts\ast}$ dual to
$\t$ is called a \textit{weight\/} of $\gl_m\ts$. 
We will regard any sequence $\mu=(\ts\mu_1\lcd\mu_m\ts)$
of length $m$ of complex numbers
as such a weight, by setting $\mu\ts(E_{\ts aa})=\mu_a$ 
for $a=1\lcd m\ts$. For any $\gl_m\ts$-module $U$,
its subspace consisting of all vectors of weight $\mu$  
is denoted by $U^\mu$. In the above display
$\n$ is the nilpotent
subalgebra of $\gl_m$ spanned by all the 
elements $E_{\ts ab}$ with $a>b\,$, while
$\np$ is spanned by all $E_{\ts ab}$ with $a<b\,$.
We will denote by $U_{\n}$ the vector space $U/\ts\n\,U$
of the {\it coinvariants\/} 
of the action of the subalgebra $\n$ on $U$.
Note that the Cartan subalgebra $\h\subset\gl_m$ acts on 
the vector space $U_{\n}\,$. 

Now consider the Verma module $M_{\ts\mu}$ of the Lie algebra
$\gl_m\,$. It can be described as the quotient of the universal enveloping 
algebra $\Ur(\ts\gl_m\ts)$ by the left ideal generated by all the elements
$E_{\ts ab}$ with $a<b$ and by the elements $E_{\ts aa}-\mu_a\ts$.
The elements of the Lie algebra $\gl_m$ act on this quotient
via left multiplication. 
Let us apply the functor $\E_N$ to the $\gl_m$-module $U=M_{\ts\mu}\ts$.
By using Proposition \ref{1.2} 
we obtain a bimodule $\E_N(M_\mu)$
of $\gl_m$ and $\H_N\ts$. For any $\la=(\ts\la_1\lcd\la_m\ts)$
consider the space 
$\E_N(\ts M_{\ts\mu})_{\ts\n}^{\ts\la}$
of those coinvariants
of this bimodule relative to $\n$ which are of the weight $\la\,$.
This space comes with an action of the algebra $\H_N\ts$. 

\begin{prop}
\label{1.3}
The $\H_N$-module $\E_N(\ts M_{\ts\mu})_{\ts\n}^{\ts\la}$
is isomorphic to the standard module $S_{\mu}^{\ts\la}\,$.
\end{prop}

\begin{proof}
By repeatedly using \cite[Theorem 1.3]{KN6} the proof %of the proposition 
reduces to its particular case when $m=1\ts$. 
In the latter case the proposition is immediate.
\end{proof}

Let us now give a counterpart \cite[Proposition 5.3.1]{AST}
of Proposition \ref{1.3}
for the Lie algebra $\sl_m$ instead of $\gl_m\,$.
We will denote by $\h$ the Cartan subalgebra of $\sl_m$ with 
the basis vectors $E_{\ts 11}-E_{\ts 22}\,\lcd 
E_{\ts m-1\ts,\ts m-1}-E_{\ts mm}\ts$. Note that $\h\subset\t\,$.
We have the triangular decomposition
\begin{equation*}
\sl_m=\n\op\h\op\np
\end{equation*}
where $\n$ and $\np$ are the same as above.
We will denote respectively by $\al$ and $\be$
the restrictions of the weights $\la$ and $\mu$ of $\gl_m$
to the subspace $\h\subset\t\,$. Thus $\al$ and $\be$ will 
be weights of $\sl_m\,$. 
Note that the
restriction of the $\gl_m\ts$-module $M_{\ts\mu}$ to the subalgebra
$\sl_m\subset\gl_m$ is isomorphic to the Verma module $M_{\ts\be}\,$,
while the central element $I\in\gl_m$ acts on $M_\mu$ as multiplication
by $\mu_1+\ldots+\mu_m\,$. Therefore by using the definition of 
%the functor 
$\F_N$ we get a corollary to Proposition~\ref{1.3}. 

\begin{cor}
\label{1.4}
The $\H_N$-module $\F_N(\ts M_{\ts\be}\ts)_{\,\ts\n}^{\,\al}$
is isomorphic to the pullback of standard module $S_{\mu}^{\ts\la}$
relative to the automorphism \eqref{shift} of the algebra\/ $\H_N$ where\/
$
f=-\,(\ts\mu_1+\ldots+\mu_m\ts)\ts/\ts m\,.
$
\end{cor}

Note that by pulling the standard module $S_{\mu}^{\ts\la}$
back through the automorphism \eqref{shift} with any $f$
we get another standard module, corresponding to the sequences
$(\ts\la_1+f\lcd\la_m+f\ts)$ and $(\ts\mu_1+f\lcd\mu_m+f\ts)$
instead of $\la$ and $\mu\,$. However, we will use
Corollary \ref{1.4} as stated.

%=============================================================================

\section{Cherednik algebras}
\label{sec:2}
\medskip

%-----------------------------------------------------------------------------

\subsection{}
\label{sec:21}

Let $\P_N=\CC[\ts x_1\ts,x_1^{-1}\lcd x_N\ts,x_N^{-1}\ts]$ be the 
ring of of Laurent polynomials in $N$ variables $x_1\lcd x_N$ 
with complex coefficients. We will denote by $\d_1\lcd\d_N$
the derivation operators in $\P_N$ relative to these variables. 
%Let $\ka\in\CC\,$.
The \textit{trigonometric Cherednik algebra\/} $\C_N$
depending on a parameter $\ka\in\CC$
is the complex associative algebra
generated by $\H_N$ and $\P_N\ts$, 
subject to the relations
$
\si\,x_p\,\si^{-1}=x_{\ts\si(p)}
$
for all $\si\in\Sym_N$ and to the commutation relations
\begin{align*}
%\label{cross8}
[\,z_{p}\,,x_q\,]&=-\,x_p\,\si_{pq}
\quad\text{for}\quad
q\neq p\,;
\\[4pt]
%\label{cross9}
[\,z_{p}\,,x_p\,]&=\ka\,x_p+\sum_{r\neq p}\,x_p\,\si_{pr}\,.
\end{align*}
We can also employ the pairwise commuting generators 
$u_1\lcd u_N\in\H_N$ 
instead of $z_1\lcd z_N\ts$; 
see the definition \eqref{zp}.
Then instead of the above displayed defining relations 
%\eqref{cross8} and \eqref{cross9} 
in $\C_N$ we get
\begin{align*}
%\label{cross5}
[\,u_{p}\,,x_q\,]&=-\,x_q\,\si_{pq}
\quad\text{for}\quad
q<p\,;
\\[4pt]
%\label{cross6}
[\,u_{p}\,,x_q\,]&=-\,x_p\,\si_{pq}
\quad\text{for}\quad
q>p\,;
\\[4pt]
%\label{cross7}
[\,u_{p}\,,x_p\,]&=\ka\,x_p+
\sum_{r<p}\,x_r\,\si_{pr}+
\sum_{r>p}\,x_p\,\si_{pr}\,.
\end{align*}
The latter set of defining relations shows that the mapping
\eqref{shift} extends to an automorphism of $\C_N$
identical in the subalgebras $\CC\ts\Sym_N$ and $\P_N\,$.
Further, according to \cite[Theorem 1.3]{EG}
multiplication in the algebra $\C_N$ yields a bijective linear~map
\begin{equation*}
%\label{trigmult}
\P_N\ot
\CC\ts\Sym_N\ot\CC[\ts u_1\lcd u_N\ts]
\,\to\,\C_N\,.
\end{equation*}
Next we will state generalizations of Proposition \ref{1.1} and
Corollary \ref{1.2}. They go back to %the work of Cherednik 
\cite{C1}.

%-----------------------------------------------------------------------------

\subsection{}
\label{sec:22}

First consider the affine Lie algebra $\glhat_m$ over the feld $\CC\,$.
We will define it
as a central extension of the {\it current Lie algebra\/}
$\gl_m\ts[\,t,t^{-1}\ts]$ by a one-dimensional complex vector space with 
a fixed basis element which will be denoted by $C\ts$. Here $t$ is 
%regarded as 
a formal variable.
Choose the basis of $\gl_m\ts[\,t,t^{-1}\ts]$ consisting of the elements
$E_{\ts cd}\,t^{\,j}$ where $c,d=1\lcd m$ whereas %the degree 
$j$ ranges over $\ZZ\ts$.
The commutators in the Lie algebra $\gl_m\ts[\,t,t^{-1}\ts]$ 
are taken pointwise so that
$$
[\,E_{\ts ab}\,t^{\,i},E_{\ts cd}\,t^{\,j}\,]=
(\ts\de_{bc}\ts E_{\ts ad}-\de_{da}\,E_{\ts cb}\ts)\,t^{\,i+j}
$$
for the basis elements. In the extended Lie algebra $\glhat_m$ 
by definition we have the relations
\begin{equation}
\label{hatcom}
[\,E_{\ts ab}\,t^{\,i},E_{\ts cd}\,t^{\,j}\,]=
(\ts\de_{bc}\,E_{\ts ad}-\de_{da}\,E_{\ts cb}\ts)\,t^{\,i+j}+
i\,\de_{\ts i,-j}\,\de_{bc}\,\de_{da}\,C\ts.
\end{equation}

We will also work with the affine Lie algebra $\slhat_m\,$.  
This is a subalgebra of $\glhat_m$ spanned by
the subspace $\sl_m\ts[\,t,t^{-1}\ts]\subset\gl_m\ts[\,t,t^{-1}\ts]$
and by the central element $C\ts$. 
Let $\hh$ be the Abelian subalgebra 
of $\slhat_m$ spanned by $C$ and by the Cartan subalgebra 
$\h\subset\sl_m\,$. The vector spaces
$$
\nt=\n\op t^{-1}\,\sl_m[\,t^{-1}\ts]
\,\quad\text{and}\ \quad
\ntp=\np\op\ts t\,\ts\sl_m[\ts t\ts]
$$
are also Lie subalgebras of $\slhat_m$ 
by the relations \eqref{hatcom}. As a vector space,
\begin{equation*}
%\label{hatdec}
\,\slhat_m=\nt\op\hh\op\ntp.
\end{equation*}

Let $V$ be any module of $\glhat_m$ such that
for any given vector in $V$, there exists a degree 
$i$ such that the subspace $t^{\,i}\,\gl_m[t]\subset\glhat_m$ 
annihilates the vector. Consider the vector space
\begin{equation}
\label{w}
W=\P_N\ot
(\CC^{\ts m})^{\ot N}\ot\ts V\ts.
\end{equation}
Due to our condition on $V$ for any $p=1\lcd N$ 
there is a well defined linear operator on $W$
\begin{equation}
\label{gap}
\sum_{i=0}^\infty\, 
\sum_{a,b=1}^m\,
x_p^{\,-i}\ot E_{\ts ab}^{\ts(p)}\ot E_{\ts ba}\,t^{\,i}\,.
\end{equation}
Here $E_{\ts ab}^{\ts(p)}$ is the operator \eqref{epab} acting on
$(\CC^{\ts m})^{\ot N}\ts$. Further, 
the symmetric group $\Sym_N$ acts on the tensor factor $\P_N$ of $W$
by permutations of the variables $x_1\lcd x_N\,$.
There is another copy of the group $\Sym_N$ acting on 
the $N$ tensor factors $\CC^{\ts m}$ of $W$ by permutation.
Using these two actions of~$\Sym_N$
for any $p=1\lcd N$ introduce the \textit{Cherednik operator} on $W$
\begin{gather}
\nonumber
\ka\,x_p\,\d_p\ot\id^{\,\ot N}\ot\id\,+\,
\sum_{r\neq p}\,\frac{x_p}{x_p-x_r}\,(1-\si_{pr})\ot\si_{pr}\ot\id\ +
\\
\label{up}
\sum_{i=0}^\infty\,
\sum_{a,b=1}^m\, 
x_p^{\,-i}\ot E_{\ts ab}^{\ts(p)}\ot E_{\ts ba}\,t^{\,i}\,.
\end{gather}

\enlargethispage{36pt}%%%%%%%%%%%%%%%%%%%%%%%%%%%%%%%%%%%%%%%%%%%%%%%%%%%%%%%%

The vector space
$
%\label{u}
\P_N\ot(\CC^{\ts m})^{\ot N}
$
can be naturally identified with the tensor product of $N$ copies
of the space $\CC^{\ts m}[\ts t\ts,t^{-1}\ts]\,$.
The latter space can be regarded as a $\glhat_m\ts$-module
where the central element 
$C$ acts as zero.
%while $D:u\,t^i\mapsto i\,u\,t^i$ for $u\in\CC^{\ts m}\ts$.
By taking the tensor product of $N$ copies of this module
with $V$ we turn the vector space 
$W$ to a $\glhat_m\ts$-module.
The element $E_{\ts cd}\,t^{\,j}\in\glhat_m$ acts on~$W$~as
\begin{equation}
\label{ecdj}
\id\ot\id^{\,\ot N}\ot E_{\ts cd}\,t^{\,j}\,+\,
\sum_{q=1}^N\,x_q^{\,j}\ot E_{\ts cd}^{\ts(q)}\ot\id\,.
\end{equation}

For any complex number $\ell\ts$, a module of the Lie algebra
$\glhat_m$ or $\slhat_m$ is said to be of \emph{level\/} $\ell$ 
if the element $C$ 
acts on this module as that complex number.
In particular, the $\glhat_m\ts$-module
$\CC^{\ts m}[\ts t\ts,t^{-1}\ts]$ used 
above is of level zero.
We can now state the main properties
of Cherednik operators on $W$ from \cite{AST,S2}.
%which are due to Arakawa, Suzuki and Tsuchiya \cite{AST}.
These properties immediately follow from \cite[Proposition 2.3]{KN6}. 

\begin{prop}
\label{2.1}
{\rm(i)}
Using the\/ $\glhat_m\ts$-module structure on $V$, 
an action of the algebra\/ $\C_N$ on the vector space $W$
is defined as follows: 
the elements $x_p\,,x_p^{-1}\in\C_N$ act via mutiplication in\/
$\CC[\ts x_1\ts,x_1^{-1}\lcd x_N\ts,x_N^{-1}\ts]\,,$
the group $\Sym_N\subset\H_N$ acts %diagonally
by simultaneous permutations of the variables $x_1\lcd x_N$ and of the $N$ 
tensor factors\/ $\CC^{\ts m}$,
and the element $z_p\in\H_N$ acts as~\eqref{up}. 
\\
{\rm(ii)} 
This action of\/ $\C_N$ on $W$ commutes with the action of 
the Lie subalgebra\/ $\gl_m\subset\glhat_m\,$.
\\[1pt]
{\rm(iii)}
If\/ $V$ has level\/ $\ka-m$ then the action of\/ $\C_N$
preserves the subspace\, 
$\nt\ts\,W\subset W$.
\end{prop}

Below is a version of this proposition in the case when
$V$ is a module not of $\glhat_m$ but only of $\slhat_m\,$,
also due to \cite{AST}. 
There for any vector in $V$ we assume the existence of 
$i$ such that the subspace $t^{\,i}\,\sl_m[t]\subset\slhat_m$
annihilates the vector. Let $I=E_{\ts 11}+\ldots+E_{\ts mm}$ as before.
By \eqref{i} an action of %the~sum
\begin{equation}
\label{gaps}
\sum_{i=0}^\infty\,
x_p^{\,-i}\ot
\Bigl(\,
\sum_{a,b=1}^m\,
E_{\ts ab}^{\ts(p)}\ot E_{\ts ba}\,t^{\,i}-
\frac1m\ \id^{\,\ot\ts N}\ot I\,t^{\,i}
\,\Bigr)
\end{equation}
can be defined on the vector space \eqref{w} by using only the
$\slhat_m\ts$-module structure of $V$. Then for every $p=1\lcd N$ we 
have a modification of the Cherednik operator \eqref{up}
on $W$,
\begin{gather}
\nonumber
\ka\,x_p\,\d_p\ot\id^{\,\ot N}\ot\id\,+\,
\sum_{r\neq p}\,\frac{x_p}{x_p-x_r}\,(1-\si_{pr})\ot\si_{pr}\ot\id\ +
\\
\label{ups}
\sum_{i=0}^\infty\,
x_p^{\,-i}\ot
\Bigl(\,
\sum_{a,b=1}^m\,
E_{\ts ab}^{\ts(p)}\ot E_{\ts ba}\,t^{\,i}-
\frac1m\ \id^{\,\ot\ts N}\ot I\,t^{\,i}
\,\Bigr)
\,.
\end{gather}
Here we use the sum \eqref{gaps} instead of 
\eqref{gap} used in %the operator 
\eqref{up}. Further, we can turn the vector space
\eqref{w} into another $\slhat_m\ts$-module by regarding
$\,\P_N\ot(\CC^{\ts m})^{\ot N}$
as $\slhat_m\ts$-module of level~zero. 

\begin{cor}
\label{2.2}
{\rm(i)}
Using the\/ $\slhat_m\ts$-module structure on $V$, 
an action of the algebra\/ $\C_N$ on the vector space \eqref{w}
can be defined as follows: 
the elements $x_p\,,x_p^{-1}\in\C_N$ act via mutiplication in\/
$\CC[\ts x_1\ts,x_1^{-1}\lcd x_N\ts,x_N^{-1}\ts]\,,$
the group $\Sym_N\subset\H_N$ acts %diagonally
by simultaneous permutations of the variables $x_1\lcd x_N$ and of the $N$ 
tensor factors\/ $\CC^{\ts m}$,
and the element $z_p\in\H_N$ acts as~\eqref{ups}. 
\\
{\rm(ii)} 
This action of\/ $\C_N$ on $W$ commutes with the action of 
the Lie subalgebra\/ $\sl_m\subset\slhat_m\,$.
\\[1pt]
{\rm(iii)}
If\/ $V$ has level\/ $\ka-m$ then the action of\/ $\C_N$
preserves the subspace\, 
$\nt\ts\,W\subset W$.
\end{cor}

By using Corollary \ref{2.2}(i) and the definition \eqref{w}, 
we define a functor
$
\A_N: V\mapsto W
$
from the category of all $\slhat_m\ts$-modules 
satisfying the annihilation condition stated just before %the display 
\eqref{gaps}. Note that the resulting 
actions of $\slhat_m$ and $\C_N$ on $W$ do not commute in general.
However, this will be our analogue for $\slhat_m$ of the functor $\F_N$
introduced in the end of Subsection~\ref{sec:12}.

%-----------------------------------------------------------------------------

\subsection{}
\label{sec:23}
Let $\la$ and $\mu$ be same sequences of length $m$ of complex numbers
as in Subsection \ref{sec:11}. Take the standard module 
$S_{\mu}^{\ts\la}$ over the algebra $\H_N\,$. 
By regarding $\H_N$ as a subalgebra of~$\C_N\ts$,
consider the induced module
$$
\Ind^{\,\C_N}_{\,\H_N}\,S_{\mu}^{\ts\la}\,.
$$
We will denote the latter module by $\widehat{S}_{\ts\mu}^{\,\la}\,$.
Its underlying vector space can be identified with 
that of $\P_N\ot S_{\mu}^{\ts\la}$ whereon
the subalgebra $\P_N\subset\C_N$ acts via multiplication
in the first tensor factor. Notice that 
by transitivity of induction
and by the definition of $S_{\mu}^{\ts\la}\,$,
the $\C_N\ts$-module 
\\[-3pt]%%%%%%%%%%%%%%%%%%%%%%%%%%%%%%%%%%%%%%%%%%%%%%%%%%%%%%%%%%%%%%%%%%%%%%
$\widehat{S}_{\ts\mu}^{\,\la}$ is isomorphic to
\begin{equation*}
\Ind_{\,\H_\nu}^{\,\C_N}\,
S_{\mu_1}^{\ts\la_1}
\ot
S_{\mu_2-1}^{\ts\la_2-1}
\ot\ldots\ot 
S_{\mu_m-m+1}^{\ts\la_m-m+1}\,.
\end{equation*}
%where $\H_\nu$ is regarded as a subalgebra of $\C_N\,$.

Now suppose that $\ell=\ka-m\,$, so that
our Corollary \ref{2.2}(iii) applies.
Regard $\la$ and $\mu$ as weights of $\gl_m\ts$.
Their restrictions to the subspace $\h\subset\t$ are denoted
by $\al$ and $\be$ respectively.
Define the weight $\beh$ of $\slhat_m$
as the element of the space dual to $\hh$
such that
\begin{equation}
\label{mut}
\beh\ts(C)=\ell
\quad\text{and}\quad
\beh\ts(X)=\be\ts(X)
\quad\textrm{for all}\quad
X\in\hh\,.
\end{equation}
Consider the Verma module $M_{\ts\beh}$ of
$\slhat_m\,$. By definition, this
is the quotient of the universal 
\\[-2pt]%%%%%%%%%%%%%%%%%%%%%%%%%%%%%%%%%%%%%%%%%%%%%%%%%%%%%%%%%%%%%%%%%%%%%%
enveloping 
algebra $\Ur(\ts\slhat_m\ts)$ by the left ideal generated by $\ntp$
and by all the elements $X-\beh\ts(X)$ where $X$ ranges over $\hh\,$.
Since the element $C\in\slhat_m$ is central, the first equality in
\eqref{mut} implies that the $\slhat_m\ts$-module $M_{\ts\beh}$ is
of level $\ell\/$. Moreover, $V=M_{\ts\beh}$ satisfies the 
annihilation condition stated just before \eqref{gaps}.
Therefore we can apply the functor $\A_N$ to this~$V$.

Further, let us define the weight $\alh$ of $\slhat_m$ similarly 
to $\beh$ and consider the space
$\A_N(\ts M_{\ts\beh}\,)_{\,\ts\nt}^{\,\alh}$
of those coinvariants of 
$\A_N(\ts M_{\ts\beh}\,)$ relative to $\nt$ 
which have the weight $\alh\,$.
This space comes with an action of the algebra $\C_N$
due to Corollary~2.2(iii). The latter action is described by
the next proposition \cite[Proposition 5.2.3]{AST}. The proof
given in \cite{AST} was different however.

\begin{prop}
\label{2.3}
The $\H_N$-module $\A_N(\ts M_{\ts\beh}\,)_{\,\ts\nt}^{\,\alh}$
is isomorphic to the pullback of $\widehat{S}_{\mu}^{\,\la}$
relative to the automorphism \eqref{shift} of the algebra\/ $\C_N$ where\/
$
f=-\,(\ts\mu_1+\ldots+\mu_m\ts)\ts/\ts m\,.
$
\end{prop}

\begin{proof}
By the transitivity of induction, the module $M_{\ts\beh}$
of the Lie algebra $\slhat_m$ is isomorphic to the module of
level $\ell$ parabolically induced from the Verma module
$M_{\ts\be}$ of $\sl_m\,$.
To define the parabolically induced module, we first extend the action
of $\sl_m$ on $M_{\ts\be}$
to the subalgebra $\p$ of $\slhat_m$ spanned by $\sl_m[\ts t\ts]$ and $C$.
Namely, we let 
the elements of $t\,\ts\sl_m[\ts t\ts]$ act on $M_{\ts\be}$
as zero, while $C$ acts as multiplication by $\ell\,$.
Then we induce the resulting action from $\p$ to $\slhat_m\,$.

Now denote by $\q$ the subspace
$t^{-1}\,\sl_m\ts[\,t^{-1}\ts]\subset\sl_m\ts[\,t,t^{-1}\ts]\,$.  
This is a Lie subalgebra of $\nt\,$,
and moreover $\nt=\n\op\q$ as a vector space.
Consider the space of coinvariants
of $\A_N(\ts M_{\ts\beh}\,)$ relative to $\q\,$.
This space comes with mutually commuting actions of 
$\C_N$ and $\sl_m\,$, see again
Corollary \ref{2.2}.
By \cite[Theorem 2.5]{KN6}
the so obtained bimodule of $\C_N$ and $\sl_m\,$
is isomorphic~to
$$
\Ind^{\,\C_N}_{\,\H_N}\,\F_N(\ts M_{\ts\be}\ts)\,.
$$
Our Proposition \ref{2.3} now follows from Corollary \ref{1.4}
and from the definition of $\widehat{S}_{\ts\mu}^{\,\la}\,$.
\end{proof}

%-----------------------------------------------------------------------------

\subsection{}
\label{sec:24}

Denote by $\T_m$ the \emph{affine Weyl group} of the Lie algebra
$\gl_m\,$. This group is generated by the elements 
$\tau_c$ where $c=0,1\lcd m-1\ts$. However, we will let the indices
of the generators $\tau_c$ run through $\ZZ\,$, assuming that 
$\tau_{c\ts+m}= \tau_c$ for $c\in\ZZ\,$.
Then the defining relations of $\T_m$~are
$$
\tau_c^{\ts 2}=1\,;
\quad
\tau_c\,\tau_{c+1}\,\tau_c=\tau_{c+1}\,\tau_c\,\tau_{c+1}\,;
\quad
\tau_c\,\tau_d=\tau_d\,\tau_c
\quad\text{for}\quad c-d\neq\pm\ts1\ \textrm{mod}\ m\,.
$$
The corresponding \emph{extended affine Weyl group} is 
generated by $\T_m$ and an element $\pi$ such~that
$$
\pi\,\tau_c=\tau_{c+1}\,\pi\,.
$$
Let us denote the extended group by $\R_m\,$.
The group $\R_m$ acts on the set $\ZZ$ by permutations of period 
$m\,$. Namely, each generator $\tau_c$ of $\T_m$ exchanges 
$c+j\,m$ with $c+1+j\,m$ for each $j\in\ZZ\,$,
leaving all other integers fixed.
The extra generator $\pi$ maps any integer $d$ to $d+1\,$.

The group $\R_m$ is $\ZZ\ts$-graded so that the element 
$\pi$ has degree one, while all elements of $\T_m$ have degree zero.
Further, the group $\R_m$ is isomorphic to the
semidirect product $\Sym_m\ltimes\ZZ^m\,$.  
We will use the isomorphism $\R_m\to\Sym_m\ltimes\ZZ^m$
defined by mapping 
$$
\pi\,\mapsto(1,0\lcd 0)\,\si_1\ldots\si_{\ts m-1}
\quad\text{and}\quad
\tau_{\ts0}\ts\mapsto(1,0\lcd 0\ts,-1)\,\si_{1m}
$$
while $\tau_c\ts\mapsto\si_c$ for $c=1\lcd m-1\,$. 
Here $\Sym_m$ and $\ZZ^m$ are regarded as subgroups of 
$\Sym_m\ltimes\ZZ^m\,$. In particular, 
here $(1,0\lcd 0)$ and $(1,0\lcd 0\ts,-1)$
are elements of $\ZZ^m\subset\Sym_m\ltimes\ZZ^m\,$. 
Via this isomorphism, the $\ZZ\ts$-grading on 
$\R_m$ defined here corresponds to that on $\Sym_m\ltimes\ZZ^m$
as defined in the Introduction.
Relative to the latter $\ZZ\ts$-grading, 
the degree of any element of
$\Sym_m$ is zero, while
the degree of any element of $\ZZ^m$ is
the sum of its $m$ components.

The group $\R_m$ acts by automorphisms 
of the Lie algebra $\glhat_m$ so that the central element $C$
is invariant, 
$$
\tau_c:E_{\ts ab}\,t^{\,i}\mapsto 
E_{\ts\tau_c(a)\ts,\ts\tau_c(b)}\,t^{\,i}
\quad\text{for}\quad
c=1\lcd m-1
$$
while
$$
\tau_{\ts0}:E_{\ts ab}\,t^{\,i}\mapsto 
E_{\ts\tau_{\ts0}(a)\ts,\ts\tau_{\ts0}(b)}\,
t^{\,i+\de_{a1}-\de_{am}-\de_{\ts b1}+\de_{\ts bm}}+\ts
\de_{i\ts0}\,\de_{ab}\,(\de_{a1}-\de_{am})\,C
$$
and 
\begin{equation*}
%\label{Epi}
\pi:
E_{\ts ab}\,t^{\,i}\mapsto 
E_{\ts a+1,\ts b+1}\,t^{\,i-\de_{am}+\de_{\ts bm}}-\ts
\de_{i\ts0}\,\de_{ab}\,\de_{am}\,C\,.
\end{equation*}
Then
$$
\pi^{-1}:
E_{\ts ab}\,t^{\,i}\mapsto 
E_{\ts a-1,\ts b-1}\,t^{\,i+\de_{a1}-\de_{\ts b1}}+\ts
\de_{i\ts0}\,\de_{ab}\,\de_{a1}\,C\,.
\hspace{20pt}
\vspace{05pt}
$$
Here we let $a,b=1\lcd m\,$.
If any of the indices of the matrix units 
appearing in the last three displayed formulas
is $0$ or $m+1\ts$, it should be then replaced respectively by $m$ or $1\,$.
%We will keep using this
%interpretation of the indices of matrix units. 

\enlargethispage{12pt}%%%%%%%%%%%%%%%%%%%%%%%%%%%%%%%%%%%%%%%%%%%%%%%%%%%%%%%%

Consider the level zero module $\CC^{\ts m}[\ts t\ts,t^{-1}\ts]$ 
of the Lie algebra $\glhat_m\,$. Let $e_1\lcd e_m$ be the standard 
basis vectors of $\CC^{\ts m}\ts$.  
The group $\R_m$ acts on the vector space 
$\CC^{\ts m}[\ts t\ts,t^{-1}\ts]$ so that 
$$
\tau_c:e_a\,t^{\,i}\mapsto 
e_{\ts\tau_c(a)}\,t^{\,i}
\quad\text{for}\quad
c=1\lcd m-1
$$
while
$$
\tau_{\ts0}:e_a\,t^{\,i}\mapsto 
e_{\ts\tau_{\ts0}(a)}\,
t^{\,i+\de_{a1}-\de_{am}}
\ \quad\text{and}\ \quad
\pi:e_a\,t^{\,i}\mapsto 
e_{\ts a+1}\,t^{\,i-\de_{am}}\,.
$$
Then
$$
\pi^{-1}:e_a\,t^{\,i}\mapsto 
e_{\ts a-1}\,t^{\,i+\de_{a1}}\,.
$$

Here we use the same interpretation of the indices of the standard basis
vectors of $\CC^m$ as of the indices of the matrix units above.
One can easily verify that the actions of $\glhat_m$ and of $\R_m$ on
$\CC^{\ts m}[\ts t\ts,t^{-1}\ts]$ extend to an action of
the crossed product algebra $\R_m\ltimes\Ur(\ts\glhat_m\ts)\,$.
This algebra is defined by the above described action
of the group $\R_m$ on $\glhat_m\,$. In the crossed product algebra, 
$$
\pi\,X\ts\pi^{-1}=\pi\ts(X)
\quad\text{for}\quad
X\in\glhat_m\,.
$$

%-----------------------------------------------------------------------------

\subsection{}
\label{sec:25}

Suppose that the $\glhat_m\ts$-module $V$ is also equipped with an 
action of the extended affine Weyl 
group $\R_m\,$. Moreover, suppose that
the actions of both $\glhat_m$ and $\R_m$ on $V$ extend to an action of 
the crossed product algebra $\R_m\ltimes\Ur(\ts\glhat_m\ts)\,$. 
By identifying the tensor product of $N$ copies
of $\CC^{\ts m}[\ts t\ts,t^{-1}\ts]$ with 
$\P_N\ot(\CC^{\ts m})^{\ot N}$
we define an action of the group $\R_m$ on the 
latter vector space, and hence on its tensor product \eqref{w}
with $V$. 

By the definition given in Subsection \ref{sec:24},
the action of the element 
$\pi\in\R_m$ on the Lie algebra~$\glhat_m$ 
preserves the subalgebra $\nt\,$. Therefore the element 
$\pi$ acts on the space $W_{\,\nt}$ of $\nt\ts$-coinvariants
of the $\glhat_m\ts$-module $W$. On the other hand,
under the assumption $\ell=\ka-m$ 
the Cherednik operator \eqref{up} %with any $p=1\lcd N$
also acts on $W_{\,\nt}\ts$ due to Proposition~\ref{2.1}.
Let us denote by $\ze_{\,p}$ the operator on $W_{\,\nt}$
corresponding to \eqref{up}.
The next property of $\ze_{\,p}$
will be crucial~for~us. 

\begin{prop}
\label{2.4}
If the\/ $\glhat_m\ts$-module $V$ has level $\ka-m$ then
for $p=1\lcd N$ we have an equality of operators on\/ %the space 
$W_{\,\nt}$
\begin{equation*}
\pi\ts\,\ze_{\,p}\,\pi^{-1}=\ze_{\,p}+\id\,.
\end{equation*}
\end{prop}

\begin{proof}
Extend the vector space \eqref{w}
by replacing its first tensor factor $\P_N$ 
by the space of all complex valued 
rational functions in the variables $x_1\lcd x_N$
with the permutation action of the symmetric group $\Sym_N\,$.
Extend the action on \eqref{w} of the element $\pi$ accordingly.
To this end, identify the tensor product 
$\P_N\ot(\CC^{\ts m})^{\ot N}$ in \eqref{w}
with the the tensor product of $N$ copies
of $\CC^{\ts m}[\ts t\ts,t^{-1}\ts]$ as above.
Then restate the definition of the action
of $\pi$ on the vector space $\CC^{\ts m}[\ts t\ts,t^{-1}\ts]$
by regarding the latter as the tensor product
$\CC^{\ts m}\ot\CC\ts[\ts t\ts,t^{-1}\ts]\,$.
 
For any $p=1\lcd N$ consider the following three
operators on the extended vector space,
\begin{gather*}
D_p=\ka\,x_p\,\d_p\ot\id^{\,\ot N}\ot\id\,,
\quad
R_{\ts p}=\sum_{r\neq p}\,\frac{x_p}{x_p-x_r}\,\si_{pr}\ot\si_{pr}\ot\id\,,
\\[2pt]
T_p=
\sum_{r\neq p}\,\frac{x_p}{x_p-x_r}\ot\si_{pr}\ot\id\ +
\sum_{i=0}^\infty\,\sum_{a,b=1}^m\, 
x_p^{\,-i}\ot E_{\ts ab}^{\ts(p)}\ot E_{\ts ba}\,t^{\,i}\,.
\end{gather*}
Then the operator \eqref{up} is the restriction of the operator
$D_p-R_{\ts p}+T_p$ to the space~\eqref{w}.

\enlargethispage{12pt}%%%%%%%%%%%%%%%%%%%%%%%%%%%%%%%%%%%%%%%%%%%%%%%%%%%%%%%%

By identifying
$\P_N\ot(\CC^{\ts m})^{\ot N}$ with 
the tensor product of $N$ copies
of $\CC^{\ts m}[\ts t\ts,t^{-1}\ts]$
and using the action of the element $\pi$ on the $p\,$th of these $N$ copies 
as defined in Subsection~\ref{sec:24},
\begin{equation}
\label{Dp}
\pi\ts\,D_p\,\pi^{-1}=D_p+\id\ot \ka\,E_{\ts11}^{\ts(p)}\ot\id\,.
\end{equation}
The action of $\pi$ on the tensor product 
$\P_N\ot(\CC^{\ts m})^{\ot N}$
commutes with the multiplication by any element of $\P_N$ in the first 
tensor factor. It also commutes with the operator $\si_{pr}\ot\si_{pr}$
for any $r\neq p\,$. Therefore
$$
\pi\ts\,R_{\ts p}\,\pi^{-1}=R_{\ts p}\,.
$$

Consider the operator $T_p\,$. In its definition, 
the summand corresponding to any $r\neq p$ can be rewritten as
$$
\sum_{i=0}^\infty\,\sum_{a,b=1}^m\,
x_p^{\,-i}\ts x_r^{\,i}\ot
E_{\ts ab}^{\ts(p)}\ts E_{\ts ba}^{\ts(r)}\ot\id\,. 
$$
Therefore $\pi\ts\,T_p\,\pi^{-1}$ is equal to 
the sum over the indices $i=0\ts,1\ts,\ldots$ and
$a,b=1\lcd m$ of 
\begin{gather*}
\sum_{r\neq p}\, 
x_p^{\,-i-\de_{am}+\de_{\ts bm}}\ts x_r^{\,i+\de_{am}-\de_{\ts bm}}\ot
E_{\ts a+1,b+1}^{\ts(p)}\ts E_{\ts b+1,a+1}^{\ts(r)}\ot\id\ +
\\[4pt]
x_p^{\,-i-\de_{am}+\de_{\ts bm}}\ot
E_{\ts a+1,b+1}^{\ts(p)}\ot
(\ts E_{\ts b+1,a+1}\,t^{\,i+\de_{am}-\de_{\ts bm}}-\ts
\de_{i\ts0}\,\de_{ab}\,\de_{\ts bm}\,\ell\,)\,.
\end{gather*}
Here we use the action of $\pi$ on $\Ur(\ts\glhat_m\ts)$ 
as defined in Subsection \ref{sec:24}.
By the definition of $T_p\,$,
the sum over the indices $i$ and $a,b$
of the expressions in the two lines displayed above equals
\begin{gather*}
T_p
\,+\,
\sum_{r\neq p}\,\sum_{a\neq m}\,
x_p\ts\,x_r^{\,-1}\ot
E_{\ts a+1,1}^{\ts(p)}\,E_{\ts 1,a+1}^{\ts(r)}\ot\id
\,-\,
\sum_{r\neq p}\,\sum_{b\neq m}\,\id\ot
E_{\ts 1,b+1}^{\ts(p)}\,E_{\ts b+1,1}^{\ts(r)}\ot\id\ +
\\[6pt]
\sum_{a\neq m}\, 
x_p\ot E_{\ts a+1,1}^{\ts(p)}\ot E_{\ts 1,a+1}\,t^{\,-1}
\,-\,
\sum_{b\neq m}\,
\id\ot E_{\ts 1,b+1}^{\ts(p)}\ot E_{\ts b+1,1}
\,-\,\id\ot\ell\,E_{\ts11}^{\ts(p)}\ot\id\,.
\end{gather*}
By adding to this result the right hand side of \eqref{Dp} and by
subtracting $R_{\ts p}\,$, we get back~the 
Cherednik operator \eqref{up} plus the sum
\begin{gather*}
\sum_{r\neq p}\,\sum_{a\neq1}\,
x_p\ts\,x_r^{\,-1}\ot
E_{\ts a1}^{\ts(p)}\,E_{\ts 1a}^{\ts(r)}\ot\id
\,-\,
\sum_{r\neq p}\,\sum_{b\neq1}\,\id\ot
E_{\ts 1b}^{\ts(p)}\,E_{\ts b1}^{\ts(r)}\ot\id\ +
\\[4pt]
\sum_{a\neq1}\, 
x_p\ot E_{\ts a1}^{\ts(p)}\ot E_{\ts 1a}\,t^{\,-1}
\,-\,
\sum_{b\neq1}\,
\id\ot E_{\ts 1b}^{\ts(p)}\ot E_{\ts b1}
\,+\,\id\ot m\,E_{\ts11}^{\ts(p)}\ot\id\,.
\end{gather*}
Here we used the equality $\ka-\ell=m\ts$ and
replaced the indices $\ts a+1,b+1\ts$ respectively by~$a,b\,$.

The sum in the last two displayed lines can be rewritten as
\begin{gather*}
\sum_{a\neq1}\,
\Bigl(\ 
\sum_{r=1}^N
x_r^{\,-1}\ot E_{\ts 1a}^{\ts(r)}\ot\id
+
\id\ot\id^{\,\ot N}\ot E_{\ts 1a}\,t^{\,-1}
\ts\Bigr)\cdot
x_p\ot E_{\ts a1}^{\ts(p)}\ot\id\ +
\\[-2pt]
\sum_{b=1}^m\,
\id\ot E_{\ts bb}^{\ts(p)}\ot\id
\,-\,
\sum_{b\neq1}\,
\Bigl(\ 
\sum_{r=1}^N
\id\ot E_{\ts b1}^{\ts(r)}\ot\id
+
\id\ot\id^{\,\ot N}\ot E_{\ts b1}
\,\Bigr)\cdot
\id\ot E_{\ts 1b}^{\ts(p)}\ot\id\,.
\end{gather*}
Here the sum over $b=1\lcd m$ is the identity operator on $W$.
For any $a\neq1$ the element $E_{\ts1a}\,t^{-1}\in\nt$ acts on $W$ as
the sum in the brackets in the first of the two lines displayed here,
see \eqref{ecdj}. Hence the whole expression in that line
vanishes on the quotient $W_{\,\nt}\,$.  
Further, for any $b\neq1$ the element $E_{\ts b1}\in\nt\ts$ acts on $W$
as the sum in the brackets in the second 
of~the two lines.
Hence the whole expression in that line
acts on $W_{\,\nt}\ts$
as the identity operator.  
\end{proof}

Below is a version of Proposition \ref{2.4} in the case
when $V$ is a module not of $\glhat_m$ but only of $\slhat_m\,$.
Here we regard $W$ as $\slhat_m\ts$-module,
and use the action of the element 
$\pi\in\R_m$ on the corresponding
space $W_{\,\nt}\ts$ of $\nt\ts$-coinvariants.
Under the assumption $\ell=\ka-m\,$, 
the modified Cherednik operator \eqref{ups} 
acts on $W_{\,\nt}\ts$ due to Corollary~\ref{2.2}.
Let us denote by $\theta_{\,p}$ the operator on $W_{\,\nt}$
corresponding to \eqref{ups}.

\begin{cor}
\label{2.5}
If the\/ $\slhat_m\ts$-module $V$ has level $\ka-m$ then
for $p=1\lcd N$ we have an equality of operators on\/ %the space 
$W_{\,\nt}$
\begin{equation*}
\pi\ts\,\theta_{\,p}\,\pi^{-1}\,=\,\theta_{\,p}\,+\,\frac{\ka}m\,\,\id\,.
\end{equation*}
\end{cor}

\noindent{\it Proof.}
The modified Cherednik operator \eqref{ups} is obtained by subtracting
from \eqref{up} the~sum
$$
\frac1m\,\ts\sum_{i=0}^\infty\,
x_p^{\,-i}\ot\id^{\,\ot\ts N}\ot I\,t^{\,i}\,.
$$
But the action of $\pi\in\R_m$ on the latter sum amounts to 
subtracting from it the operator 
$$
\frac1m\,\,\id\ot\id^{\,\ot\ts N}\ot C\,,
$$
see Subsection \ref{sec:24}.
Since the module $V$ has level $\ka-m\,$, 
our Proposition \ref{2.4} implies that
$$
\pi\ts\,\theta_{\,p}\,\pi^{-1}\,=\,
\theta_{\,p}\,+\,\frac{\ka-m}m\,\,\id\,+\,\id\,=\,
\theta_{\,p}\,+\,\frac{\ka}m\,\,\id\,\,.
\eqno{\square}
$$

%=============================================================================

\section{Zhelobenko operators}
\label{sec:3}
\medskip

%-----------------------------------------------------------------------------

\subsection{}
\label{sec:31}

Let us denote by $\th$ the subalgebra of $\glhat_m$
with the basis vectors $C$ and $E_{\ts 11}\lcd E_{\ts mm}\,$.
Note that $\th$ contains the subalgebra $\hh$ of $\slhat_m\,$.
Consider the action of the extended affine~Weyl group $\R_m$
on $\glhat_m$ defined in Subsection \ref{sec:24}.
This action preserves the subalgebra $\th\subset\glhat_m\,$.
By definition, we have $\pi\ts(C)=C$ and 
$\tau\ts(C)=C$ for all $\tau\in\T_m\,$. Further, we have
\begin{gather*}
\pi\ts(E_{\ts dd})=E_{\ts d+1,d+1}
\quad\text{for}\quad
1\le d<m\,,
\quad
\pi\ts(E_{\ts mm})=E_{\ts 11}-C\,,
\\[6pt]
\tau_{\,0}\ts(E_{\ts 11})=E_{\ts mm}+C\,,
\quad
\tau_{\,0}\ts(E_{\ts dd})=E_{\ts dd}
\quad\text{for}\quad
1<d<m\,,
\quad
\tau_{\,0}\ts(E_{\ts mm})=E_{\ts 11}-C
\end{gather*}
while the generators
$\tau_{\,1}\lcd\tau_{\,m-1}$ act on the basis vectors 
$E_{\ts 11}\lcd E_{\ts mm}$ naturally, that is 
by transpositions of the indices $1\lcd m\,$.
Note that here we also have
$$
\pi^{-1}\ts(E_{\ts dd})=E_{\ts d-1,d-1}
\quad\text{for}\quad
1<d\le m\,,
\quad
\pi^{-1}\ts(E_{\ts11})=E_{\ts mm}+C\,.
$$

We will also use the action of the group 
$\R_m$ on the vector space $\th^{\,\ast}$, dual to the above described
action on $\th\,$. To describe the dual action explicitly,
let $C^{\,\ast}$ and $E_{\ts 11}^{\,\ast}\lcd E_{\ts mm}^{\,\ast}$ be
the basis vectors of $\th^{\,\ast}$ dual to our chosen basis vectors
of $\th\,$. Then 
$$
\pi\ts(C^{\,\ast})=C^{\,\ast}+E_{\ts 11}^{\,\ast}
\quad\text{and}\quad
\tau_{\,0}\ts(C^{\,\ast})=C^{\,\ast}+
E_{\ts 11}^{\,\ast}-E_{\ts mm}^{\,\ast}\,,
$$
\begin{gather*}
\pi\ts(E_{\ts dd}^{\,\ast})=E_{\ts d+1,d+1}^{\,\ast}
\quad\text{for}\quad
1\le d<m\,,
\quad
\pi\ts(E_{\ts mm}^{\,\ast})=E_{\ts11}^{\,\ast}\,,
\\[6pt]
\tau_{\,0}\ts(E_{\ts 11}^{\,\ast})=E_{\ts mm}^{\,\ast}\,,
\quad
\tau_{\,0}\ts(E_{\ts dd}^{\,\ast})=E_{\ts dd}^{\,\ast}
\quad\text{for}\quad
1<d<m\,,
\quad
\tau_{\,0}\ts(E_{\ts mm}^{\,\ast})=E_{\ts 11}^{\,\ast}
\end{gather*}
while the generators 
$\tau_{\,1}\lcd\tau_{\,m-1}$ leave %the vector 
$C^{\,\ast}$ invariant and act on the vectors 
$E_{\ts 11}^{\,\ast}\lcd E_{\ts mm}^{\,\ast}$ %naturally, that is 
by transpositions of the indices $1\lcd m\,$.

%Similarly to the definition \eqref{mut}, 
Now for any given $\ell\in\CC$
and for any weight $\mu\in\t^{\ts\ast}$ define an element 
$\muh\in\th^{\,\ast}$ by setting
\begin{equation*}
\muh\ts(C)=\ell
\quad\text{and}\quad
\muh\ts(X)=\mu\ts(X)
\quad\textrm{for all}\quad
X\in\t\,.
\end{equation*}
Equivalently,
\begin{equation*}
%\label{muh}
\muh=\ell\,C^{\,\ast}+
\mu_1\ts E_{\ts 11}^{\,\ast}+\ldots+\mu_m\ts E_{\ts mm}^{\,\ast}\,.
\end{equation*}
Then
$$
\pi\ts(\ts\muh\ts)=
\ell\,C^{\,\ast}+
(\ts\mu_m+\ell\,)\ts E_{\ts 11}^{\,\ast}+
\mu_1\ts E_{\ts 22}^{\,\ast}+\ldots+
\mu_{m-1}\ts E_{\ts mm}^{\,\ast}\,,
$$
\vspace{-10pt}
$$
\tau_{\,0}\ts(\ts\muh\ts)=
\ell\,C^{\,\ast}+
(\ts\mu_m+\ell\,)\ts E_{\ts 11}^{\,\ast}+
\mu_2\ts E_{\ts 22}^{\,\ast}+\ldots+
\mu_{m-1}\ts E_{\ts m-1,m-1}^{\,\ast}+
(\ts\mu_1-\ell\,)\ts E_{\ts mm}^{\,\ast}\,.
$$
In particular, for any given $\ell\in\CC$ the action of %the group 
$\R_m$ on $\th^{\,\ast}$
preserves the set of weights of the form $\muh\,$. 
So we get an action of %the group 
$\R_m$ on the set of sequences 
of length $m$ of complex~numbers,
\begin{align*}
\pi&:(\ts\mu_1\lcd\mu_m\ts)\mapsto
(\ts\mu_m+\ell\,,\mu_1\ts\lcd\mu_{m-1}\,)
\\[4pt]
\tau_{\,0}&:(\ts\mu_1\lcd\mu_m\ts)\mapsto
(\ts\mu_m+\ell\,,\mu_2\ts\lcd\mu_{m-1}\ts,\ts\mu_1-\ell\,)
\end{align*}
while
$\tau_{\,1}\lcd\tau_{\,m-1}$ naturally act on these sequences
by transpositions of the indices $1\lcd m\,$.
Note that via the isomorphism $\R_m\to\Sym_m\ltimes\ZZ^m$
chosen in Subsection \ref{sec:24},
the same action of $\R_m$ on the sequences
can be obtained by letting the elements of
$\Sym_m$ act by permutations, while the elements of $\ZZ^m$ 
act by addition of the respective elements of~$\,\ell\,\ZZ^m\subset\CC^m\,$.

We will also employ the \emph{shifted action\/} of the group $\R_m$ on
$\th^{\,\ast}$. It is defined 
by adding 
\begin{equation}
\label{rho}
m\,C^{\,\ast\!}
-E_{\ts 11}^{\,\ast}-2\ts E_{\ts 22}^{\,\ast}-\ldots-m\,E_{\ts mm}^{\,\ast}
\end{equation}
to the elements of $\,\th^{\,\ast}$, then applying the above described
action of $\R_m\,$, and then subtracting \eqref{rho}.
We will employ the symbol $\comp$ to denote the shifted action.
Put $\ep_{\ts0}=E_{\ts mm}^{\,\ast}-E_{\ts 11}^{\,\ast}$
while $\ep_{\ts c}=E_{\ts cc}^{\,\ast}-E_{\ts c+1,c+1}^{\,\ast}$
for $c=1\lcd m-1\,$. Note that then 
\begin{equation}
\label{plusep}
\tau_c\,\comp\,\muh=\tau_c\ts(\,\muh+\ep_{\ts c}\ts)
\quad\text{for}\quad
c=0\ts,1\lcd m-1\,.
\end{equation}

For any given $\ell$ 
the shifted action of $\R_m$ on $\th^{\,\ast}$
preserves the set of weights of the form $\muh\,$. 
Hence we get a shifted action of %the group 
$\R_m$ on the set of sequences 
of length $m$ of complex numbers. 
We will use use the symbol $\comp$ to denote it as well.
%In this notation, 
Then for any sequence $\mu=(\ts\mu_1\lcd\mu_m\ts)$
\begin{gather*}
\pi\,\comp\,\mu=
(\ts\mu_m+\ell+1\ts,\mu_1+1\ts\lcd\mu_{\ts m-1}+1\ts)\,,
\\[4pt]
\tau_{\,0}\,\comp\,\mu=
(\ts\mu_m+\ell+1\ts,\mu_2\ts\lcd\mu_{\ts m-1}\ts,\mu_1-\ell-1\ts)\,,
\\[4pt]
\tau_c\,\comp\,\mu=
(\ts\mu_1\lcd\mu_{\ts c-1}\ts,
\mu_{\ts c+1}-1\,,\mu_{c}+1\ts,
\mu_{\ts c+2}\ts\lcd\mu_m\ts)
\quad\text{for}\quad c=1\lcd m-1\,. 
\end{gather*}

Note that via our isomorphism $\R_m\to\Sym_m\ltimes\ZZ^m\,$,
the same shifted action of the group $\R_m$ on the sequences
can be obtained by using the last displayed formula and 
by letting the elements of the subgroup $\ZZ^m\subset\Sym_m\ltimes\ZZ^m\,$ 
act by addition of the respective elements of $\,\ka\,\ZZ^m$
where $\ka=\ell+m\,$.
Indeed, because the group $\R_m$ is generated by
$\tau_{\,1}\lcd\tau_{\,m-1}$ and $\pi\,$, it suffices to 
check the coincidence of two actions of the element $\pi$ only.
Its image under the isomorphism $\R_m\to\Sym_m\ltimes\ZZ^m$ is
the product $(1,0\lcd 0)\,\si_1\ldots\si_{\ts m-1}\,$,
see Subsection \ref{sec:24}. But by our definition of
the shifted action of the group $\Sym_m\ltimes\ZZ^m$ 
on the sequences we have 
\begin{align*}
\si_1\ldots\si_{\ts m-1}\,\comp\,\mu
&=
(\ts\mu_m-m+1\ts,\mu_1+1\ts\lcd\mu_{\ts m-1}+1\ts)\,,
\\[4pt]
(1,0\lcd 0)\,\si_1\ldots\si_{\ts m-1}\,\comp\,\mu
&=
(\ts\mu_m-m+\ka+1\ts,\mu_1+1\ts\lcd\mu_{\ts m-1}+1\ts)
\\[4pt]
&=
(\ts\mu_m+\ell+1\ts,\mu_1+1\ts\lcd\mu_{\ts m-1}+1\ts)\,.
\end{align*}

%-----------------------------------------------------------------------------

\subsection{}
\label{sec:32}

Consider the tensor product of $N$ copies
of the $\slhat_m\ts$-module $\CC^{\ts m}[\ts t\ts,t^{-1}\ts]$ of level zero.
In Subsection~\ref{sec:22} we identified the vector space
of this tensor product with $\P_N\ot(\CC^{\ts m})^{\ot N}$. 
Put 
\begin{equation*}
%\label{Br}
\Br=\P_N\ot(\CC^{\ts m})^{\ot N}\ot\Ur(\ts\slhat_m\ts)\,.
\end{equation*}
Following \cite{KNV} 
we will regard $\Br$ as bimodule over the associative algebra
$\Ur(\ts\slhat_m\ts)$ by setting
\begin{equation*}
\label{XPA}
X\,(P\ot A)=X\ts P\ot A+P\ot X\ts A
\quad\text{and}\quad
(P\ot A)\,X=P\ot A\,X
\end{equation*}
for $X\in\slhat_m$ while $P\in\P_N\ot(\CC^{\ts m})^{\ot N}$
and $A\in\Ur(\ts\slhat_m\ts)\,$. So the left module
structure on $\Br$ is defined by regarding
$\Ur(\ts\slhat_m\ts)$ as a module over itself via left 
multiplication, and then taking its tensor product with the module
$\P_N\ot(\CC^{\ts m})^{\ot N}$ by using the standard comultiplication on
$\Ur(\ts\slhat_m\ts)\,$. 
The right module structure on $\Br$ is defined 
by using only the right multiplication in the
tensor factor $\Ur(\ts\slhat_m\ts)$ of $\Br\,$.
We will also use the \emph{adjoint action} of $\Ur(\ts\slhat_m\ts)$ 
on $\Br\,$. Here
$$
\ad_{\ts X}(P\ot A)=X\,(P\ot A)-(P\ot A)\,X=
X\ts P\ot A+P\ot [\ts X\ts,A\ts]\,.
$$

The action of the %extended affine Weyl 
group $\R_m$ on the Lie algebra $\glhat_m$ 
preserves the subalgebra $\slhat_m\,$.
By again identifying the tensor product of $N$ copies
of $\CC^{\ts m}[\ts t\ts,t^{-1}\ts]$ with 
$\P_N\ot(\CC^{\ts m})^{\ot N}\ts$,
we get an action of the group $\R_m$ on the 
former vector space, and hence on the vector space~of~$\Br\,$.

Consider the universal enveloping algebra $\Hr$
of the Abelian Lie algebra $\hh\subset\slhat_m\,$.
Let $\Hb$ be the ring of fractions of the commutative
algebra $\Hr$ with the set of denominators generated by
\begin{equation}
\label{denset}
\{\,E_{\ts aa}-E_{\ts bb}+i\,C+j\ |\ 
1\le a<b\le m\ \,\text{and}\ \,i\,,j\in\ZZ\,\ts\}\,.
\end{equation}
The elements of this ring can also be
regarded as rational functions on the vector space~$\hh^{\ts\ast}\ts$. 
The elements of $\Hr\subset\Hb$
are then regarded as polynomial functions on $\hh^{\ts\ast}\ts$.
Further, let $\Ub(\ts\slhat_m\ts)$ be the ring of fractions of the algebra
$\Ur(\ts\slhat_m\ts)$ with the same set of denominators.

Let us denote
\begin{equation*}
%\label{Bb}
\Bb=\P_N\ot(\CC^{\ts m})^{\ot N}\ot\Ub(\ts\slhat_m\ts)\,.
\end{equation*}
Using the right multiplication in the algebra $\Ub(\ts\slhat_m\ts)\ts$, 
the right action of $\Ur(\ts\slhat_m\ts)$ on $\Br$ extends to
a right action of $\Ub(\ts\slhat_m\ts)$ on $\Bb\ts$. 
To extend the left action of $\Ur(\ts\slhat_m\ts)$ 
on $\Br$ to
a right action of $\Ub(\ts\slhat_m\ts)$ on $\Bb\ts$,
note that the vector space of $\Br$ has a basis of elements $Y$
such that for any $a,b$ and $i$ as in \eqref{denset}
there exists $k\in\ZZ$ also depending on $Y$, such that
$$
\ad_{\ts E_{\ts aa}-E_{\ts bb}+i\,C}\,(Y)=k\,Y\ts.
$$
Then we set
$$
(\ts E_{\ts aa}-E_{\ts bb}+i\,C+j\ts)^{\ts-1}\,Y=
Y(\ts E_{\ts aa}-E_{\ts bb}+i\,C+j+k\ts)^{\ts-1}\,.
\hspace{-24pt}
$$
Hence the vector space $\Bb$ becomes a bimodule over the algebra 
$\Ub(\ts\slhat_m\ts)\,$.

The action of the %extended affine Weyl 
group $\R_m$ on the Lie algebra $\glhat_m$ 
preserves the subalgebra $\hh\subset\slhat_m\,$. Moreover, 
the resulting action of $\R_m$ on $\Hr$ preserves the set of denominators 
generated by \eqref{denset}. So the action of %the group 
$\R_m$ extends from $\Br$ to $\Bb\ts$. We will use the extended 
action~later.

%-----------------------------------------------------------------------------

\subsection{}
\label{sec:33}

The Lie algebra $\slhat_m$ is generated by the elements 
$$
E_{\ts0}=E_{\ts m1}\,t\,,
\quad
F_{\ts0}=E_{\ts 1m}\,t^{-1}\,,
\quad
H_{\ts0}=C-E_{\ts 11}+E_{\ts mm}
$$
and %by the elements
$$
E_{\ts c}=E_{\ts c,c+1}\,,
\quad
F_{c}=E_{\ts c+1,c}\,,
\quad
H_c=E_{\ts cc}-E_{\ts c+1,c+1}
\quad\text{where}\quad c=1\lcd m-1\,.
\hspace{-20pt}
$$
For each $c=0\ts,1\lcd m-1$ the elements $E_c\,,F_c\,,H_c$
span a subalgebra of $\slhat_m$
isomorphic to the Lie algebra $\mathfrak{sl}_{\ts2}\,$.
We will also use the element $\ep_{\ts c}\in\th^{\,\ast}$
defined in Subsection~\ref{sec:31}.

Take the vector spaces $\Br$ and $\Bb$
introduced in Subsection \ref{sec:32}.
For $c=0\ts,1\lcd m-1$ define a linear map $\xi_{\ts c}:\Br\to\Bb$
by setting
\begin{equation}
\label{q1}
\xi_{\ts c}\,(\ts Y)=Y+\,\sum_{n=1}^\infty\,\,
(\ts n\ts!\,H_c^{\ts(n)}\ts)^{-1}\ts E_{\ts c}^{\,n}
\ad_{\ts F_c}^{\,n}(\ts Y)
\end{equation}
for any $Y\in\Br\ts$. Here
$$
H_c^{\ts(n)}=H_c(H_c-1)\ldots(H_c-n+1)
$$
and we take the $n\,$th power of the adjoint operator
corresponding to the element $F_c\in\slhat_m\,$.
For any given %element 
$Y\in\Br$ only finitely many terms of the sum 
\eqref{q1} differ from zero, so the map $\xi_{\ts c}$ is well defined.
The definition \eqref{q1} and the next proposition
go back to \cite[Section~2\ts]{Z}.
By using the left action of the Lie subalgebra $\nt\subset\slhat_m\,$,
introduce the vector subspaces
$$
\Jr=\nt\,\ts\Br\ts\subset\ts\Br
\quad\text{and}\quad
\Jb=\nt\,\ts\Bb\ts\subset\ts\Bb\,.
$$

\begin{prop}
\label{3.1}
For any $X\in\hh$ and\/ $Y\in\Br$ we have
\begin{align}
\label{q11}
\xi_{\ts c}\ts(\ts X\ts Y\ts)&\in
(\ts X+\ep_c(X))\,\ts\xi_{\ts c}(\ts Y)\ts+\ts\Jb\ts,
\\[2pt]
\label{q12}
\xi_{\ts c}\ts(\ts Y X\ts)&\in\,
\xi_{\ts c}(\ts Y)\ts(\ts X+\ep_c(X))\ts+\ts\Jb\ts.
\end{align}
\end{prop}

\noindent\emph{Proof.}
It suffices to verify the properties \eqref{q11} and \eqref{q12}
only for $X=H_c$ and for all $X\in\hh$ such that $\ep_{\ts c}(X)=0\,$.
In the latter case we have the relations 
$[\ts E_c\,,X\ts]=[\ts F_c\,,X\ts]=0$ in $\slhat_m\,$.
Then $\xi_{\ts c}\ts(\ts X\ts Y\ts)=X\,\xi_{\ts c}\ts(Y)$
and $\xi_{\ts c}\ts(\ts Y X\ts)=\xi_{\ts c}\ts(Y)\,X$
by \eqref{q1}. Hence we get \eqref{q11} and \eqref{q12}.

For $X=H_c$ the proof of \eqref{q11}
is based on the following commutation relations in the subalgebra of 
$\Ur(\slhat_m\ts)$ generated by the three elements 
$E_{\ts c}\,,F_{\ts c}\,,H_c\,$:
for any $n=1,2,\ts\ldots$
\begin{equation}
\label{ad}
[\,E_{\ts c}^{\,n}\ts,H_c\,]=
-\,2\,n\,E_{\ts c}^{\,n}
\quad\text{and}\quad
[\,E_{\ts c}^{\,n}\ts,F_c\,\ts]=
n\,(H_c-n+1)\,E_{\ts c}^{\,n-1}\ts.
\end{equation}

Let us use the symbol $\ts\,\equiv\,$ to indicate equalities in 
$\Bb$ modulo the subspace $\Jb\ts$.
By~\eqref{ad}, %and the inclusion $\Hb\,F_c\ts\subset\ts\Jb\ts$,
for any element $Y\in\Br$ we get
$$
\xi_{\ts c}\,(\ts H_c\,Y)
\,\equiv\,(\ts H_c+2\,)\,\ts\xi_{\ts c}\,(\ts Y)
\,=\,(\ts H_c+\ep_c(H_c))\,\ts\xi_{\ts c}\,(\ts Y)\,.
$$
Here the relation $\ts\,\equiv\,$ is obtained as
in the beginning of  
the proof of \cite[Proposition~3.1]{KN1}. By following another calculation,
as given in the end of the proof of \cite[Proposition~3.1]{KN1}, we get
\begin{equation*}
\xi_{\ts c}\,(\,YH_c\ts)
\,\equiv\,\xi_{\ts c}\,(\ts Y)\ts(\ts H_c+2\,)
\,=\,\xi_{\ts c}\,(\ts Y)\ts(\ts H_c+\ep_c(H_c))\,.
\eqno{\square}
\end{equation*}

%-----------------------------------------------------------------------------

\subsection{}
\label{sec:34}

\smallskip\medskip
The property \eqref{q11} allows us to define %for $c=0,1\lcd m-1$
a linear map 
$
\bar\xi_{\ts c}:\Bb\to\Bb\ts/\ts\Jb
$
by setting
$$
\bar\xi_{\ts c}\ts(\,YA\,)=\xi_{\ts c}\ts(\ts Y)\,Z\ts+\ts\Jb
\quad\ \text{for}\quad
A\in\Hb
\quad\text{and}\quad
Y\in\Br
$$
where the element $Z\in\,\Hb$ is obtained from $A$
by regarding it as a rational function on 
the dual vector space $\hh^{\ts\ast}\ts$,
and then adding $\ep_c$ to the argument of that rational function.
Recall that in the end of Subsection \ref{sec:32} we defined an action of
the extended affine Weyl group $\R_m$ on the vector space $\Bb\ts$. 
For any index $c=0\ts,1\lcd m-1$ 
consider the image~$\tau_c\ts(\ts\Jb\ts)\subset\Bb\ts$. 
%The next proposition goes back~to~\cite{Z}~too. 
%like Proposition \ref{3.1} does.

\begin{prop}
\label{3.2}
We have\/ $\tau_c\ts(\ts\Jb\,)\subset\ker\ts\bar\xi_{\ts c}\,$.
\end{prop} 

\begin{proof}
Note that $\tau_c\ts(F_c)=E_{\ts c}\,$.
If $c>0$ then let $\nt_{\ts c}$ be the subspace of $\slhat_m$ spanned
by all the elements $E_{\ts ab}\,t^{\,i}$ where $i<0\,$, and by
those elements $E_{\ts ab}$ where $a>b$ but $(a\ts,b)\neq(c+1\ts,c)\,$. 
Further, let $\nt_{\,0}$ be the subspace of $\slhat_m$ spanned
by the elements $E_{\ts ab}$ where $a>b\,$, and by those elements
$E_{\ts ab}\,t^{\,i}$ where $i<0$ but $(a\ts,b\ts,i)\neq(1\ts,m\ts,-1)\,$.
Then for any $c=0\ts,1\lcd m-1$ the image
$\tau_c\ts(\ts\Jb\,)\subset\Bb$ is spanned by the subspaces
$\nt_{\ts c}\ts\Bb$ and $E_{\ts c}\ts\Bb\ts$.

By using the relations \eqref{hatcom} one can check
that that the subspace $\nt_{\ts c}\subset\slhat_m$ 
is preserved by 
the adjoint action of the elements $E_c\,,F_c\,,H_c\,$. %on $\slhat_m\,$.
So we have $\xi_{\ts c}\ts(X\ts Y)\in\Jb$ for 
any $X\in\nt_c$ and any $Y\in\Br\ts$, see \eqref{q1}. 
To prove Proposition \ref{3.2} it remains to show that 
$\xi_{\ts c}\ts(\ts E_{\ts c}\,Y\ts )\in\Jb$ for any $Y\in\Br\ts$.
By using the relations \eqref{ad},
this can be shown by the same calculation as in the proof
of \cite[Proposition 3.2]{KN1}.
\end{proof}

Proposition \ref{3.2} allows us to define for $c=0\ts,1\lcd m-1$ 
a linear map
$
\,\eta_{\,c}:\,\Bb\ts/\ts\Jb\,\to\Bb\ts/\ts\Jb
$
as the composition $\bar\xi_{\ts c}\,\tau_c$ 
applied to the elements of $\Bb$ which are
taken modulo the subspace $\Jb\ts$. 
This definition also goes back to \cite{Z},
and we will call 
$$
\eta_{\,0},\eta_{\,1}\lcd\eta_{\,m-1}
$$
the \textit{Zhelobenko operators} on $\Bb\ts/\ts\Jb\,$.
The next proposition states their key property,
for its proof see \cite[Section 6]{KO}.
%\cite[Lemma 4.5, Proposition 6.1]{KO}
Like in the beginning of Subsection~\ref{sec:24}, 
%for the generators $\tau_c$ of the group $\T_n\,$,
here we will let the indices $c$ of the operators
$\eta_{\,c}$ run through $\ZZ\,$, assuming that 
$\eta_{\,c+m}=\eta_{\,c}$ for all $c\in\ZZ\,$.

\begin{prop}
\label{3.3}
The operators\/ $\eta_{\,0},\eta_{\,1}\lcd\eta_{\,m-1}$
on\/ $\Bb\ts/\ts\Jb$ satisfy the affine braid relations
$$
\eta_{\,c}\,\eta_{\,c+1}\,\eta_{\,c}=
\eta_{\,c+1}\,\eta_{\,c}\,\eta_{\,c+1}\,;
\quad
\eta_{\,c}\,\eta_{\,d}=\eta_{\,d}\,\eta_{\,c}
\quad\text{for}\quad c-d\neq\pm\ts1\ \text{\rm mod}\ m\,.
$$
\end{prop}

\begin{cor}
\label{3.4}
For any reduced decomposition $\tau=\tau_{\ts c}\ts\ldots\ts\tau_{\ts d}$ 
in\/ $\T_m$ the composition
$\eta_{\,c}\ts\ldots\ts\eta_{\,d}$
of operators on\/ $\Bb\ts/\ts\Jb$ 
does not depend on the choice of the decomposition of\/ $\tau\ts$.
\end{cor}

By the definition given in Subsection \ref{sec:24},
the action of the element 
$\pi\in\R_m$ on the Lie algebra~$\slhat_m$ maps 
$E_c\,,F_c\,,H_c\,$ respectively to 
$E_{\ts c+1}\,,F_{\ts c+1}\,,H_{\ts c+1}\,$.
If the index $c+1$ here
is $m\ts$, it should be then replaced by $0\,$. 
Furthermore, the action of the element $\pi$ on $\slhat_m$ 
preserves the subalgebra $\nt\,$.
Hence the action of $\pi$ on $\Bb$
determines its action on the quotient $\Bb\ts/\ts\Jb\,$.
It now follows from the definition \eqref{q1} that 
on $\Bb\ts/\ts\Jb$ we have
\begin{equation}
\label{pieta}
\pi\,\eta_{\,c}=\eta_{\,c+1}\,\pi\,.
\end{equation}

%-----------------------------------------------------------------------------

\subsection{}
\label{sec:35}

Using the right action of the Lie subalgebra $\ntp\subset\slhat_m\,$,
introduce the vector subspaces
$$
\Jpr=\Br\,\ts\ntp\ts\subset\ts\Br
\quad\text{and}\quad
\Jpb=\Bb\,\ts\ntp\ts\subset\ts\Bb\ts.
$$
For any $c=0\ts,1\lcd m-1$
consider the image $\tau_c\ts(\ts\Jpb\ts)\subset\Bb\ts$.

\begin{prop}
\label{3.5}
We have\/ 
$\bar\xi_{\ts c}\ts(\,\tau_c\ts(\Jpb\ts))\subset\Jb+\Jpb\,$.
\end{prop}

\begin{proof}
Note that $\tau_c\ts(E_{\ts c})=F_c\,$.
If $c>0$ then let $\ntp_{\ts c}$ be the subspace of $\slhat_m$ spanned
by all the elements $E_{\ts ab}\,t^{\,i}$ where $i>0\,$, and by
those elements $E_{\ts ab}$ where $a<b$ but $(a\ts,b)\neq(c\ts,c+1)\,$. 
Further, let $\ntp_{\,0}$ be the subspace of $\slhat_m$ spanned
by all the elements $E_{\ts ab}$ where $a<b\,$, and by those elements
$E_{\ts ab}\,t^{\,i}$ where $i>0$ but $(a\ts,b\ts,i)\neq(m\ts,1\ts,1)\,$.
Then for any $c=0\ts,1\lcd m-1$ the image
$\tau_c\ts(\ts\Jpb\,)\subset\Bb$ is spanned by the subspaces
$\Bb\,\ntp_{\ts c}$ and $\Bb\ts F_c\,$.

By using the relations \eqref{hatcom} one can check
that that the subspace $\ntp_{\ts c}\subset\slhat_m$ 
is preserved by 
the adjoint action of the element $F_c\,$. %on $\slhat_m\,$.
Hence we have $\xi_{\ts c}\ts(X\ts Y)\in\Jpb$ for 
any $X\in\ntp_c$ and any $Y\in\Br\ts$, see the definition \eqref{q1}.
Further, note that 
$\xi_{\ts c}\ts(\ts YF_c\ts)=\xi_{\ts c}\ts(\ts Y)\,F_c$
for any $Y\in\Br\,$, because 
$\ad_{\ts F_c}(\ts YF_c\ts)=\ad_{\ts F_c}(\ts Y)\,F_c\,$.
The proof of Proposition~\ref{3.5}
can be now completed by showing that here
$\xi_{\ts c}\ts(\ts Y)\ts F_c\in\Jb\,$. 
By using \eqref{ad},
the latter inclusion is obtained by the same calculation as in the proof
of \cite[Proposition 3.5]{KN1}.
\end{proof}

Proposition \ref{3.5} implies that for every $c=0\ts,1\lcd m-1$
the Zhelobenko operator $\eta_{\,c}$ determines a linear map
$$
\Bb\ts/\ts(\ts\Jb+\Jpb\ts)\,\to\,\Bb\ts/\ts(\ts\Jb+\Jpb\ts)\,.
$$
Recall that $\ell=\ka-m$ by 
an assumption made in Subsection \ref{sec:23}.
Denote by $\Ir$ the subspace
$\Br\,(\ts C-\ell\,)\subset\Br\,$.
Similarly, denote by $\Ib$ the subspace
$\Bb\,(\ts C-\ell\,)\subset\Bb\,$. Because the element $C\in\slhat_m$
is central, the Zhelobenko operator $\eta_{\,c}$ also
determines a linear map
\begin{equation}
\label{BB}
\Bb\ts/\ts(\ts\Jb+\Jpb\ts+\Ib\,)\,\to\,\Bb\ts/\ts(\ts\Jb+\Jpb+\Ib\,)\,.
\end{equation}

Observe that 
the vector space $\Bb\ts/\ts(\ts\Jb+\Jpb+\Ib\,)$
coincides with the space of $\nt\,$-coinvariants
of the $\slhat_m\,$-module \eqref{w} where
the tensor factor $V$ is the \emph{universal} Verma module
of level $\ell\,$. 
Namely,
here $V$ is the quotient of the universal enveloping algebra 
$\Ur(\ts\slhat_m\ts)$ by
the left ideal generated by $\ntp$ and by the element $C-\ell\,$.
This $V$ satisfies the annihilation condition stated
just before \eqref{gaps}. By applying Corollary \ref{2.2} to this 
$\slhat_m\,$-module $V$, 
we define an action of the Cherednik algebra $\C_N$ on  
the quotient vector space $\Bb\ts/\ts(\ts\Jb+\Jpb+\Ib\,)\,$.

Note that the action of the element $\pi$ on 
$\Bb$ also determines a linear map \eqref{BB}.
For $\ka\neq0$ this map does not commute with the 
action of $\C_N\,$, see Corollary \ref{2.5}.
However, we still have the following theorem.  
This theorem is the principal result of the present article.

\begin{theorem}
\label{3.6}
For $c=0\ts,1\lcd m-1$ and\/ $\ell=\ka-m$ the linear map \eqref{BB}
determined by the Zhelobenko operator $\eta_{\,c}$
commutes with the action of the algebra $\C_N\,$.
\end{theorem}

\noindent{\it Proof.}
First consider the linear map \eqref{BB}
determined by the Zhelobenko operator $\eta_{\,c}$ for any $c>0\,$.
This map commutes with the action of the algebra $\C_N$
by the definition \eqref{q1} of
corresponding operator $\xi_{\,c}:\Br\to\Bb\tts$, see Corollary 2.2(ii).
Here we also use the observation that for any $c>0$ the action of the 
element $\tau_c$ on $\Br$ commutes with multiplications
by the variables $x_1\lcd x_N$ and with permutations of these variables
in the tensor factor $\P_N$ of $\Br\,$, 
commutes with permutations of the $N$ tensor factors $\CC^m$ of $\Br\,$,
and commutes with~\eqref{ups} if \eqref{ups}
is regarded as an operator on $\Br$ using the 
left multiplication by elements of $\slhat_m$ in the 
tensor factor $\Ur(\slhat_m)$ of $\Br\,$.

Let us consider the linear map \eqref{BB}
determined by the Zhelobenko operator $\eta_{\,0}\,$.
We can write $\eta_{\,0}=\pi^{\,-1}\,\eta_{\,1}\,\pi\,$,
see the end of Subsection \ref{sec:34}.
The action of the element $\pi\in\R_m$
on $\Br$ commutes with multiplication
by the variables $x_1\lcd x_N$ in the tensor factor $\P_N$ of $\Br\,$,
and also commutes with simultaneous permutations
of these variables and of the corresponding
$N$ tensor factors $\CC^m$ of $\Br\,$.
By using the argument from the previous paragraph when $c=1\,$,
and by applying Corollary \ref{2.5} when
$V$ is the universal Verma module of level $\ell\,$,
we can now complete the proof of our theorem.
In particular, if we denote simply by $\up$ the linear map \eqref{BB}
determined by the Zhelobenko operator $\eta_{\,1}\,$, 
then for any index $p=1\lcd N$ we obtain
$$
\pi^{\,-1}\,\up\,\,\pi\,\,\theta_{\,p}=
\pi^{\,-1}\,\up\,\Bigl(\ts\theta_{\,p}+\frac{\ka}m\,\ts\id\,\Bigr)\,\pi=
\pi^{\,-1}\ts\Bigl(\ts\theta_{\,p}+\frac{\ka}m\,\ts\id\,\Bigr)\,\up\,\,\pi=
\theta_{\,p}\,\,\pi^{\,-1}\,\up\,\,\pi\,.
\eqno{\square}
$$

\medskip

The group $\T_m$ generated by 
$\tau_0\ts,\tau_1\ts\lcd\tau_{m-1}$
can be regarded as the Weyl group of the affine Lie algebra
$\slhat_m\,$. The quotient of %the group 
$\R_m$ by the relation
$\pi^{\ts m}=1$ can be then regarded as the extended 
Weyl group of $\slhat_m\,$. These two facts underline our
definition of the operators 
$\eta_{\,0}\ts,\eta_{\,1}\ts\lcd\eta_{\,m-1}\,$. 
In the next section we will apply Theorem \ref{3.6}
%in the situation 
when the universal Verma module $V$ of level $\ell$
appearing above is replaced by the usual Verma module $M_{\ts\beh}$
of $\slhat_m\,$. 

%=============================================================================

\section{Intertwining operators}
\label{sec:4}
\medskip

%-----------------------------------------------------------------------------

\subsection{}
\label{sec:41}

Using the right action of the Lie subalgebra $\hh\subset\slhat_m\,,$
consider the vector subspace
$$
\Br\,(\ts E_{\ts aa}-E_{\ts bb}-\mu_{\ts a}+\mu_{\ts b}\ts)\subset\Br
\quad\text{where}\quad
1\le a<b\le m\,.
$$
This subspace depends on the weight $\mu\in\t^{\ts\ast}$
via its restriction $\be$ to $\h\subset\t\,\ts$.
Let $\Ir_{\,\beh}$ be the sum of all these subspaces and of
the subspace $\Ir\subset\Br$ introduced 
just before stating Theorem~\ref{3.6}.
As an $\slhat_m\ts$-module, %of the Lie algebra $\slhat_m\,$,
the quotient $\Br\ts/\ts(\,\Jpr+\Ir_{\,\beh}\ts)$ can be identified with
\begin{equation}
\label{ANM}
\P_N\ot(\CC^{\ts m})^{\ot N}\ot M_{\ts\beh}\,=\,
\A_N(\ts M_{\ts\beh}\,)\,.
\end{equation}
Here $\slhat_m$ acts on the quotient via the left action 
of the algebra $\Ur(\ts\slhat_m\ts)$ on its bimodule $\Br\,$.

Now suppose that the sequence $(\ts\mu_1\lcd\mu_m\ts)$ of complex numbers
obeys the conditions
\begin{equation}
\label{mucon}
\mu_{\ts a}-\mu_{\ts b}\,\notin\,\ZZ+\ell\,\ZZ
\quad\text{for}\quad
1\le a<b\le m\,.
\end{equation}
Note that then the sequence $\la=(\ts\la_1\lcd\la_m\ts)$ 
obeys the same conditions, since $\la_{\ts a}-\mu_{\ts a}\in\ZZ$
for $a=1\lcd m$
by our assumption. %made in Subsection \ref{sec:11}. 
Similarly to $\Ir_{\,\beh}\,\ts$, denote by $\Ib_{\,\beh}$
the sum of all subspaces
$$
\Bb\,(\ts E_{\ts aa}-E_{\ts bb}-\mu_{\ts a}+\mu_{\ts b}\ts)\subset\Bb
\quad\text{where}\quad
1\le a<b\le m\,,
$$
and of the subspace $\Ib\,$.
As a module of $\slhat_m\,$,
the quotient $\Bb\ts/\ts(\,\Jpb+\Ib_{\,\beh}\,)$
can be also identified with \eqref{ANM}.
The quotient $\Bb\ts/\ts(\,\Jb+\Jpb+\Ib_{\,\beh}\,)$
is then identified with the space 
$\A_N(\ts M_{\ts\beh}\,)_{\,\ts\nt}\ts$
of $\,\nt\ts$-coinvariants of \eqref{ANM}.

The shifted action of the affine Weyl group $\R_m$ on $\th^{\,\ast}$
determines an action of $\R_m$ on~$\hh^{\,\ast}\,$.
We use the same symbol $\comp$ to denote the latter action. Then
by \eqref{plusep} we have the equality
\begin{equation*}
%\label{beplus}
\tau_c\,\comp\,\beh=\tau_c\ts(\,\beh+\ep_{\ts c}\ts)
\quad\text{for}\quad
c=0\ts,1\lcd m-1\,.
\end{equation*}
Here the summand $\ep_c$ defined in Subsection \ref{sec:31} 
is regarded as a linear function on 
the vector space $\hh$ by restriction from $\th\,$. 
Due to \eqref{q12} and to the last displayed equality,
we have %an inclusion
$$
\bar\xi_{\ts c}\ts(\,\tau_c\ts(\,\ts\Ib_{\,\beh}\,))
\ts\subset\ts\Jb+\Ib_{\,\ts\tau_c\ts\comp\ts\beh}\,\,.
$$
Therefore the Zhelobenko operator $\eta_{\,c}$ 
defined in Subsection \ref{sec:34} determines a linear map
\begin{equation}
\label{BJJI}
\Bb\ts/\ts(\,\Jb+\Jpb+\Ib_{\,\beh}\,)
\,\to\,
\Bb\ts/\ts(\,\Jb+\Jpb+\Ib_{\,\ts\tau_c\ts\comp\ts\beh}\,)\,,
\end{equation}
see also Subsection \ref{sec:35}.
Via the identifications described above, \eqref{BJJI} becomes
a linear map
$$
\A_N(\ts M_{\ts\beh}\,)_{\,\ts\nt}\ts
\,\to\,
\A_N(\ts M_{\,\tau_c\ts\comp\ts\beh}\,)_{\,\ts\nt}\,.
$$
Then by \eqref{q11} the restriction of %the map 
\eqref{BJJI}
to subspace of vectors of weight $\alh\ts$ becomes
a linear~map
\begin{equation}
\label{ANMT}
\A_N(\ts M_{\ts\beh}\,)_{\,\ts\nt}^{\,\alh}\ts
\,\to\,
\A_N
(\ts M_{\,\tau_c\ts\comp\ts\beh}\,)_{\,\ts\nt}^{\,\tau_c\ts\comp\ts\alh}
\,\ts.
\end{equation}

Now consider the action of the element $\pi\in\R_m$ on $\Bb$.
This action preserves the subspaces $\Jb$ and $\Jpb$ of $\Bb$. 
Further, by the definition of the subspace $\Ib_{\,\beh}$ of 
$\Bb$ we have 
$$
\pi\ts(\,\Ib_{\,\beh}\,)=\Ib_{\,\pi\ts(\,\beh\,)}\,.
$$ 
Here we employ the usual, not shifted action of %the group 
$\R_m$ on $\th^{\,\ast}$.
But we also have 
$\pi\ts(\,\beh\,)=\pi\ts\comp\ts\beh\,$, see again Subsection \ref{sec:31}.
Hence the action of $\pi$ on $\Bb$ determines a linear map
\begin{equation}
\label{BJJIP}
\Bb\ts/\ts(\,\Jb+\Jpb+\Ib_{\,\beh}\,)
\,\to\,
\Bb\ts/\ts(\,\Jb+\Jpb+\Ib_{\,\ts\pi\ts\comp\ts\beh}\,)\,.
\end{equation}
Via the identifications described above, \eqref{BJJIP} becomes
a linear map
$$
\A_N(\ts M_{\ts\beh}\,)_{\,\ts\nt}\ts
\,\to\,
\A_N(\ts M_{\,\pi\ts\comp\ts\beh}\,)_{\,\ts\nt}\,.
$$
Then the restriction of %the map 
\eqref{BJJIP}
to the subspace of vectors of weight $\alh\ts$ becomes
a linear~map
\begin{equation}
\label{ANMP}
\A_N(\ts M_{\ts\beh}\,)_{\,\ts\nt}^{\,\alh}\ts
\,\to\,
\A_N
(\ts M_{\,\pi\ts\comp\ts\beh}\,)_{\,\ts\nt}^{\,\pi\ts\comp\ts\alh}
\,\ts.
\end{equation}

%-----------------------------------------------------------------------------

\subsection{}
\label{sec:42}

In Subsection \ref{sec:23} we did already assume that $\ell=\ka-m\,$.
Under this assumption, by Theorem \ref{3.6}
the action of the algebra $\C_N$
on the vector space $\Bb\ts/\ts(\ts\Jb+\Jpb+\Ib\,)$
commutes with the linear map \eqref{BB}
determined by the Zhelobenko
operator $\eta_{\,c}$ for $c=0\ts,1\lcd m-1\,$.
Further, since
the action of $\C_N$ on $\Bb$ commutes with the right action of %the algebra 
$\Ub(\ts\slhat_m\ts)\,$,
the algebra $\C_N$ acts on the source and target vector spaces
of the linear map \eqref{BJJI}. Moreover, the map \eqref{BJJI}
intertwines these two actions.

It follows that the linear map \eqref{ANMT} corresponding to \eqref{BJJI}
is also $\C_N\ts$-intertwining. This is because the action
of $\C_N$ on the vector space \eqref{ANM} 
as defined in Subsection \ref{sec:22} corresponds to the action of
$\C_N$ on %the quotient 
$\Br\ts/\ts(\,\Jpr+\Ir_{\,\beh}\ts)\,$.
Similar correspondence holds for %the weight 
$\tau_c\,\comp\,\beh$ instead of $\beh\,$.

We can replace the source and target $\C_N\ts$-modules in
\eqref{ANMT} by their isomorphic modules,
using Proposition \ref{2.3}. 
The value of %the parameter 
$f$ appearing in that proposition is the same for
the sequence $\mu$ and for the sequence $\tau_c\,\comp\,\mu\,$
instead of $\mu\,$,
see the end of Subsection~3.2.~Hence 
our replacement modules
in \eqref{ANMT} will be pullbacks of respectively
$\widehat{S}_{\mu}^{\,\la}$ and
$\widehat{S}_{\ts\tau_c\ts\comp\ts\mu}^{\ts\,\tau_c\ts\comp\ts\la}$
relative to the same 
automorphism \eqref{shift}.
By applying the inverse of this automorphism,
the Zhelobenko operator $\eta_{\,c}$ now determines
an $\C_N\ts$-intertwining linear map
\begin{equation}
\label{SS}
\widehat{S}_{\mu}^{\,\la}
\,\to\,
\widehat{S}_{\ts\tau_c\ts\comp\ts\mu}^{\ts\,\tau_c\ts\comp\ts\la}\,.
\end{equation}

Note that because $\ell=\ka-m\,$, the conditions \eqref{mucon}
on the sequence $\mu$ can be restated as
$$
\mu_{\ts a}-\mu_{\ts b}\,\notin\,\ZZ+\ka\,\ZZ
\quad\text{for}\quad
1\le a<b\le m\,.
$$
Under the latter conditions
%and under an extra condition~that $\ka\neq0\,$, 
both the source and target $\C_N\ts$-modules in \eqref{SS} 
are irreducible by \cite[Proposition 2.4.3]{AST}.
Hence any intertwining linear map between them is
unique up to a factor
from $\CC\ts$. In the next subsection we will determine 
this scalar factor for the intertwining 
map determined by the Zhelobenko operator~$\eta_{\,c}\,$. 

Now consider the map \eqref{BJJIP} and the corresponding 
map \eqref{ANMP}, which are determined by the action of the element 
$\pi\in\R_m$ on $\Bb$. The map \eqref{ANMP} is not $\C_N\ts$-intertwining
unless $\ka=0\,$, see Corollary \ref{2.5}. 
However, it will become intertwining
if we replace the target $\C_N\ts$-module in \eqref{ANMP} by its pullback 
via the automorphism \eqref{shift} of $\C_N$ where $f=\ka\ts/m\,$.

\newpage%%%%%%%%%%%%%%%%%%%%%%%%%%%%%%%%%%%%%%%%%%%%%%%%%%%%%%%%%%%%%%%%%%%%%%

We can now replace the source and target $\C_N\ts$-modules of the latter
intertwining operator by their isomorphic modules, 
again using Proposition \ref{2.3}. 
The source module can be replaced by
the pullback of ${\widehat{S}}_{\ts\mu}^{\,\la}$
relative to the automorphism \eqref{shift}
where $f=-\,(\ts\mu_1+\ldots+\mu_m\ts)\ts/\ts m\,$.
\\[-2pt]
The target module here can be replaced by the pullback of
$\widehat{S}_{\,\pi\ts\comp\mu}^{\ts\,\pi\ts\comp\la}$
relative to \eqref{shift} where 
$$
f=-\,(\ts\mu_1+\ldots+\mu_m+\ell+m\ts)\ts/\ts m+\ka\ts/m
=-\,(\ts\mu_1+\ldots+\mu_m\ts)\ts/\ts m\,.
$$
Here we used the formula for $\pi\ts\comp\ts\mu$ 
given in Subsection \ref{sec:31}. Since the values of $f$ for
the source and the target replacement modules are the same, 
the action of $\pi$ on $\Bb$ now
determines a $\C_N\ts$-intertwining operator
\begin{equation}
\label{SSP}
\widehat{S}_{\mu}^{\ts\la}
\to
\widehat{S}_{\,\pi\ts\comp\mu}^{\ts\,\pi\ts\comp\la}\,.
\end{equation}
By the irreducibility of the source and of target 
induced $\C_N\ts$-modules here, the latter operator
must coincide with the intertwining operator from \cite{S1}
up to a scalar multiplier.

Any element of the group $\R_m$ has a reduced decomposition
of the form $\pi^{\,g}\ts\tau_{\ts c}\ts\dots\ts\tau_{\ts d}$
where $g$ is the degree of this element relative to the 
$\ZZ\ts$-grading defined in %the beginnning of 
Subsection \ref{sec:24}.
By using Corollary \ref{2.4} and the relation \eqref{pieta}, 
the composition of linear maps
$$£
\pi^{\,g}\ts\eta_{\,c}\ts\dots\ts\eta_{\,d}:
\,\Bb\ts/\ts\Jb\,\to\Bb\ts/\ts\Jb
$$
now determines a $\C_N\ts$-intertwining operator
$
\widehat{S}_{\mu}^{\ts\la}
\to
\widehat{S}_{\,\pi^{\,g}\ts\tau_{\ts c}\ts
\dots
\ts\tau_{\ts d}\ts\comp\mu}^{\ts\,\pi^{\,g}\ts\tau_{\ts c}\ts
\dots
\ts\tau_{\ts d}\ts\comp\la}\,.
$
It is defined as a composition of intertwining operators
of the form \eqref{SS},\eqref{SSP}
corresponding to generators of $\R_m\,$.
This is the intertwining operator mentioned in the
Introduction, where $\om$ is now the image of the element 
$\pi^{\,g}\ts\tau_{\ts c}\ts\dots\ts\tau_{\ts d}$
under the isomorphism of groups $\R_m\to\Sym_m\ltimes\ZZ^m\,$.
Indeed, under this isomorphism the 
shifted actions of the two groups on $\la$ and $\mu$
correspond to each other.

%-----------------------------------------------------------------------------

\subsection{}
\label{sec:43}

In this subsection we will provide explicit formulas for
the linear maps \eqref{BJJI} and \eqref{BJJIP} determined by the
Zhelobenko operator $\eta_{\,c}$ with $c=0,1\lcd m-1$ and by 
the action of the element $\pi\in\R_m\,$. 
For $1\le a_1\lcd a_N\le m$ and $i_1\lcd i_N\in\ZZ$
consider the element
$$
Y_{\,a_1\ldots\ts a_N}^{\,i_1\ldots\ts i_N}=
x_1^{\ts i_1}\ldots x_N^{\ts i_N}\ot
e_{a_1}\ot\ldots\ot e_{a_N}\ot1\in\Br\,.
$$
Due to the Poincar\'e\ts-Birkhoff\ts-Witt theorem for the universal 
enveloping 
algebra $\Ur(\ts\slhat_m\ts)\ts$, the images of all these elements in
the source quotient vector space of the maps \eqref{BJJI} and \eqref{BJJIP}
make a basis in this quotient. Now suppose that the numbers
$1\lcd m$ occur respectively $\nu_1\lcd\nu_m$ times in 
the sequence $a_1\lcd a_N\,$. 
Then the image of the element
$\,Y_{\,a_1\ldots\ts a_N}^{\,i_1\ldots\ts i_N}$ in the 
\\[-2pt]
source quotient has the weight $\alh$ relative to the left $\hh\ts$-module
structure on the quotient. 
%In particular, then we have 
%$\nu_c=\la_{\ts c}-\mu_{\ts c}$ for $c=1\lcd m\,$.

\begin{prop}
\label{4.1}
For $1\le c<m\ts$, the Zhelobenko operator $\eta_{\,c}$
maps the image~of\/ $Y_{\,a_1\ldots a_N}^{\,i_1\ldots i_N}$
in the source quotient in \eqref{BJJI} to the image 
of next sum of elements of\/ $\Br$~in the \text{target~quotient\ts:}
$$
\sum_{h=0}^{\min\ts(\nu_c\ts,\ts\nu_{c+1})}
\!\!\!\!\!
\sum_{\ b_1,\,\ldots\ts,\,b_N}\,\,
h\ts!\,(\ts\la_{\ts c+1}-\la_{\ts c}-1\ts)\,\,
Y_{\,b_1\ldots\ts b_N}^{\,i_1\ldots\ts i_N}\,\,
\prod_{s=0}^h\,\,
\frac1{\mu_{\ts c+1}-\la_{\ts c}+s-1}
$$ 
where\/ $b_1\lcd b_N$ is a sequence obtained from\/ $a_1\lcd a_N$
by changing\/ $\nu_c-h$ terms\/ $c$ to\/ $c+1\ts$, and also changing\/ 
$\nu_{c+1}-h$ terms\/ $c+1$ to\/ $c\,$.
\end{prop}

\begin{proof}
Applying the action of $\tau_c$ on $\Br$ to the element 
$Y_{\,a_1\ldots\ts a_N}^{\,i_1\ldots\ts i_N}$ amounts to replacing
every $c$ in the sequence $a_1\lcd a_N$ by $c+1\ts$, and other way round.
Let $d_1\lcd d_N$ be the sequence so obtained.
Apply the operator $\xi_{\ts c}$ to the resulting element of $\Br$
by using the definition \eqref{q1}. 
Note that for any $n=0,1,2,\ldots$ we have an equality in $\Br$
$$
E_{\ts c}^{\,n}
\ad_{\ts F_c}^{\,n}
(\ts Y_{\,d_1\ldots\ts d_N}^{\,i_1\ldots\ts i_N}\ts)=
x_1^{\ts i_1}\ldots x_N^{\ts i_N}\ot
E_{\ts c}^{\,n}
(\ts
F_c^{\ts n}\ts(\ts e_{d_1}\ot\ldots\ot e_{d_N})\ot1
\ts)\,.
$$
Modulo the subspace $\Jpr\subset\Br\,,$
the element of $\Br$ displayed here at the right hand side equals
$$
x_1^{\ts i_1}\ldots x_N^{\ts i_N}\ot
E_{\ts c}^{\,n}
F_c^{\ts n}\ts(\ts e_{d_1}\ot\ldots\ot e_{d_N})\ot1\,.
$$
In its turn, the last displayed element equals the sum of 
elements of the form 
$Y_{\,b_1\ldots\ts b_N}^{\,i_1\ldots\ts i_N}$
taken with certain multiplicities.
Namely, here the multiplicity is
the number of ways the sequence $b_1\lcd b_N$ can be obtained from
$d_1\lcd d_N$ by consecutively replacing any 
$n$ occurences of $c$ by $c+1\ts$,
and then consecutively replacing any
$n$ occurences of $c+1$ by $c$ in the result.

Denote by $h$ 
the number of those terms of the sequence
$d_1\lcd d_N$ which are equal to $c\,$, 
but change to $c+1$ in the
sequence $b_1\lcd b_N\,$. Obviously $h\le\nu_{c+1}\,$.
When passing from $d_1\lcd d_N$ to $b_1\lcd b_N$ as above,
the numbers of occurences of $1\lcd m$ remain the same.
Therefore $h$ is also
the number of those terms of the sequence
$d_1\lcd d_N$ which are equal to $c+1\,$, 
but change to $c$ in $b_1\lcd b_N\,$.
Hence $h\le\nu_c\,$. The above stated multiplicity is not zero only
if $h\le n\le\nu_{c+1}\,$. In this case it is equal to
$$
\frac{\hspace{16pt}\,n\ts!\,n\ts!\,(\ts\nu_{c+1}-h\ts)\ts!}
{(\ts n-h\ts)\ts!\,(\ts\nu_{c+1}-n\ts)\ts!}\,\,.
$$
It follows that modulo the subspace $\Jpb\subset\Bb\ts$, the element 
$
\xi_{\ts c}\ts(\ts Y_{\,d_1\ldots\ts d_N}^{\,i_1\ldots\ts i_N}\ts)
$
of $\Bb$ equals the sum
\begin{equation}
\label{hg}
\sum_{h=0}^{\ts\min(\nu_c\ts,\ts\nu_{c+1})}
\!\!\!\!\!
\sum_{\ b_1,\,\ldots\ts,\,b_N}\,\,\,
\sum_{n=h}^{\nu_{c+1}}\,\,\,
\frac{\,\,\,n\ts!\,(\ts\nu_{c+1}-h\ts)\ts!}
{(\ts n-h\ts)\ts!\,(\ts\nu_{c+1}-n\ts)\ts!\,H_c^{\ts(n)}}\,\,
(\ts Y_{\,b_1\ldots\ts b_N}^{\,i_1\ldots\ts i_N}\ts)
\end{equation}
where $b_1\lcd b_N$ range as stated in Proposition \ref{4.1}.
Here the sum over $n=h\lcd\nu_{c+1}$
can be computed by the Gauss formula for %the value of 
the hypergeometric function $\mathrm{F}\ts(\ts u,v,w\ts;z\ts)$ at $z=1\,$,
$$
\mathrm{F}\ts(\ts u,v,w\ts;1\ts)=
\frac
{\,\mathrm\Gamma(w)\,\mathrm\Gamma(w-u-v)}
{\,\mathrm\Gamma(w-u)\,\mathrm\Gamma(w-v)}
$$
which is valid for any $u,v,w\in\CC$ where $w\neq 0,-1,\ldots$ 
and $\mathrm{Re}\ts(\ts w-u-v\ts)>0\,$. By setting
$u=h-\nu_{c+1}$ and $v=h+1$ in that formula, we obtain
an equality of rational functions in $w$ 
$$
\sum_{n=h}^{\nu_{c+1}}\,\,\,
\frac{(-1)^{n-h}\,n\ts!\,(\ts\nu_{c+1}-h\ts)\ts!}
{\,h\ts!\,(\ts n-h\ts)\ts!\,(\ts\nu_{c+1}-n\ts)\ts!\,}
\prod_{s=0}^{n-h-1}\frac1{w+s}
\,\,\,=\prod_{s=0}^{\nu_{c+1}-h-1}\frac{w-h+s-1}{w+s}\ .
$$
Replacing the complex variable $w$ by the element $h-H_c\in\Ur(\h)$
in this equality, the sum of the fractions in \eqref{hg} 
taken over the indices $n=h\lcd \nu_{c+1}$ equals the product
$$
h\ts!\,\,\,\,\prod_{s=0}^{h-1}\,\,\,\frac1{H_c-s}
\,\,\,\,\cdot\!\!
\prod_{s=0}^{\nu_{c+1}-h-1}\!\frac{H_c-s+1}{H_c-s-h}
\,\,\,=\,\,\,
h\ts!\,(H_c+1)
\,\,\,\prod_{s=0}^{h}\,\,\,
\frac1{H_c-\nu_{c+1}+s+1}\ .
$$

On the other hand, we also know that
the image of the element \eqref{hg} of $\Bb$
in the target quotient in %the linear map 
\eqref{BJJI}
has weight $\tau_c\,\comp\,\alh\ts$ 
relative to the left $\hh\ts$-module
structure on the quotient. So when computing the image,
we can replace the element $H_c\in\h$ in \eqref{hg}
by the weight~value 
$$
(\ts\tau_c\,\comp\,\alh\ts)\,(H_c)=\la_{\ts c+1}-\la_{\ts c}-2\,.
$$ 
Using the relation $\la_{\ts c+1}-\nu_{c+1}=\mu_{\ts c+1}\ts$, 
we now complete the proof of Proposition \ref{4.1}.
\end{proof}

\enlargethispage{48pt}

Note that by definition, the action of $\pi$ on $\Br$
maps the element $Y_{\,a_1\ldots a_N}^{\,i_1\ldots i_N}$ to %the element
$Y_{\,b_1\ldots\ts b_N}^{\,j_1\ldots\ts j_N}$
where $b_p=a_p+1$ and $j_p=i_p-\de_{\ts a_p\ts m}$ 
for every $p=1\lcd N\,$.
By using this observation along with the relation
\eqref{pieta} for $c=m\,$,
the next result can be derived from Proposition \ref{4.1}.
It can also be obtained by directly
following the arguments employed in the proof of that proposition.

\begin{prop}
\label{4.2}
For $c=0\ts$,
the Zhelobenko operator $\eta_{\,0}$
maps the image~of\/ $Y_{\,a_1\ldots a_N}^{\,i_1\ldots i_N}$
in the source quotient in \eqref{BJJI} to the image 
of the next sum of elements of\/ $\Br$~in the target \text{quotient\ts:}
$$
\sum_{h=0}^{\min\ts(\nu_{\ts1},\ts\nu_m)}
\!\!\!\!\!
\sum_{\ b_1,\,\ldots\ts,\,b_N}\,\,
h\ts!\,(\ts\la_1-\la_{\ts m}-\ell-1\ts)\,\,
Y_{\,b_1\ldots\ts b_N}^{\,j_1\ldots\ts j_N}\,\,
\prod_{s=0}^h\,\,
\frac1{\mu_1-\la_{\ts m}-\ell+s-1}
$$ 
where 
$$
j_p=i_p+\de_{\ts a_p\ts 1}-\de_{\,b_p\ts 1}=
i_p-\de_{\ts a_p\ts m}+\de_{\,b_p\ts m}
\quad\text{for}\quad
p=1\lcd N
$$ 
whereas\/ $b_1\lcd b_N$ is a sequence obtained from\/ $a_1\lcd a_N$
by changing\/ $\nu_{\ts1}-h$ terms\/ $1$ to $m\ts$, and also changing\/ 
$\nu_{m}-h$ terms\/ $m$ to\/ $1\,$.
\end{prop}

%=============================================================================

%=============================================================================

\end{document}